\documentclass[pdflatex,sn-basic]{sn-jnl}

\usepackage{graphicx}%
\usepackage{amsmath,amssymb,mathtools}
\usepackage{physics}
\usepackage{amsthm}
\usepackage{hyperref}
\usepackage{cleveref}
\usepackage{bm}
\usepackage{tikz}
\usepackage{booktabs}
\usepackage[normalem]{ulem}
\usetikzlibrary{shapes.geometric, calc, fit, backgrounds, positioning}
\usepackage{cancel}
\usepackage{caption}
\usepackage{makecell}
\usepackage{siunitx}
\usepackage{arydshln}
\usepackage{enumitem}

\theoremstyle{thmstyleone}%
\numberwithin{equation}{section}

\newtheorem{example}{Example}[section]
\newtheorem{theorem}{Theorem}[section]
\newtheorem{proposition}{Proposition}[section]
\newtheorem{lemma}{Lemma}[section]
\newtheorem{definition}{Definition}[section]
\newtheorem{remark}{Remark}[section]
\newtheorem{corollary}{Corollary}[section]

\newcommand{\diag}{\mathrm{diag}}
\newcommand{\trans}{\intercal}
\newcommand{\R}{\mathbb{R}}
\newcommand{\C}{\mathbb{C}}
\newcommand{\Z}{\mathbb{Z}}
\newcommand{\DomX}{\mathcal{X}}

\newcommand{\constS}{\Sigma}
\newcommand{\Sym}[1]{\operatorname{Sym}(#1)}
\newcommand{\Skew}[1]{\operatorname{Skew}(#1)}
\newcommand{\Ug}[1]{\mathsf U(#1)}

\newcommand{\SymHN}{\Sym{\size}}
\newcommand{\SkewHN}{\Skew{\size}}

\newcommand{\OO}{\mathsf O}
\newcommand{\ON}{\OO(\size)}

\newcommand{\map}{\mathfrak P}
\newcommand{\cod}{\mathcal Y}
\newcommand{\xx}{{\bm x}}
\newcommand{\yy}{\bm y}
\newcommand{\QO}{Q_0}
\newcommand{\graph}{G}
\newcommand{\tree}{T}
\newcommand{\graphA}{\graph(A)}
\newcommand{\Aset}{\{A^\mu\}_{\mu=1}^\dof}
\newcommand{\tildeAset}{\{\tilde{A}^\mu\}_{\mu=1}^\dof}
\newcommand{\edge}{E}
\newcommand{\edgepair}[2]{\{#1, #2\}}
\newcommand{\Verti}{V}
\newcommand{\bgamma}{\bm{\gamma}}
\newcommand{\dof}{M}
\newcommand{\size}{N}
\newcommand{\calA}{\mathcal{A}}

\newcommand{\eframe}{{\bm u}}

\newcommand{\setspec}{\texttt{spec}}
\newcommand{\setpp}{\texttt{+2p-loops}}
\newcommand{\settwo}{\texttt{+2-loops}}
\newcommand{\setfc}{\texttt{+fund.cyc.}}

\newcommand{\Fset}{\mathcal F}

\newcommand{\rmse}{\mathrm{RMSE}}
\newcommand{\dist}{\mathrm{dist}}

\newcommand{\smalltimes}{\mathbin{\scalebox{0.7}{$\!\times\!$}}}
\begin{document}


\title{Stable Recovery of Matrix Gauge Classes \\ 
        from Pointwise Invariants}

\author[1]{\fnm{Dexuan} \sur{Zhou}}\email{zhoudexuan@mail.bnu.edu.cn}

\author[1]{\fnm{Huajie} \sur{Chen}}\email{chen.huajie@bnu.edu.cn}

\author[2]{\fnm{Bernie} \sur{Hsu}}\email{berniehsu@math.ubc.ca}

\author[2]{\fnm{Christoph} \sur{Ortner}}\email{ortner@math.ubc.ca}

\affil*[1]{%
  \orgdiv{School of Mathematical Sciences},
  \orgname{Beijing Normal University},
  \orgaddress{\city{Beijing}, \postcode{100875}, \country{China}}%
}

\affil[2]{%
  \orgdiv{Department of Mathematics},
  \orgname{University of British Columbia},
  \orgaddress{\city{Vancouver}, \postcode{V6T 1Z2}, 
  \state{BC}, 
  \country{Canada}}%
}

\abstract{
A parameterized matrix family $x\mapsto H(x)$ on a configuration domain is determined by its physical content only up to a constant orthogonal change of basis. This gauge ambiguity is intrinsic to data-driven Hamiltonian models, such as tight-binding parameterizations, reduced-order electronic structure methods, or excited-state models. It raises a basic inverse problem: what observations of $H(x)$ suffice to identify the family up to this gauge? The pointwise spectrum is incomplete already for linear families on $\mathbb{R}$. Here, we prove that, under natural non-degeneracy and connectivity assumptions, augmenting the spectrum with loop products of the coupling matrices in the instantaneous eigenframe yields a complete invariant and that inversion is stable. We support the theory with numerical experiments.
}




\maketitle

\section{Introduction}
\label{sec:introduction}
Matrix-valued families $H(\xx)$ arise throughout physics and molecular modelling whenever an effective operator, represented by an $\size \times \size$ matrix, depends on continuous parameters $\xx \in \R^\dof$. Canonical examples, and our main motivation, are discretized or reduced-order electronic structure models where $\xx$ describes a configuration of a molecule or material and $H(\xx)$ is the Hamiltonian in a chosen basis: its eigenvalues are the orbital or band energies, and its eigenvectors the corresponding states~\citep{martin2020electronic}.

Increasingly, such operators are not constructed from first principles but learned from data. This is by now a standard task in scientific machine learning: neural-network parameterizations are fitted to {\it ab initio} datasets in Hamiltonian learning \citep{li2022deep, gong2023general, qian2026equivariant, zhang2022equivariant}, pseudopotential construction \citep{lin2025deep,cances2016existence}, excited-state modelling \citep{axelrod2022excited,barrett2025transferable}, and the construction of friction operators for coarse-grained or nonadiabatic dynamics \citep{sachs2025machine,zhang2020symmetry}.

A matrix $H(\xx)$ is the representation of an operator in a basis, and the basis itself carries no physical content. A global change of basis, applied uniformly in $\xx$,
\begin{equation*}
    H(\xx)\mapsto \QO^{\trans}H(\xx)\QO,
    \qquad \QO\in\ON,
\end{equation*}
alters every entry of $H(\xx)$ while leaving every downstream prediction of interest --- spectra, densities, forces, response functions --- unchanged. 
Most of the approaches cited above nevertheless fix a basis at the outset and fit the model by pointwise matching of matrix entries. The object recovered is then one representative rather than the operator family itself, and the loss is not invariant under a change of basis: part of the residual measures the basis rather than the model \citep{wang2024universal}. 

We take up the opposite question. Can an operator family be learned from basis-independent observables alone?
In practice this would decouple the basis in which the training data are generated from the basis in which the surrogate model is fitted: for example in Hamiltonian learning the data could be produced by a plane-wave DFT code, while the surrogate model can be expressed in an atomic-orbital-inspired basis resulting in a size extensive model.

An affirmative answer requires basis-independent quantities that are rich enough to determine the family up to a global change of basis (a gauge). The pointwise spectrum is the obvious candidate, but it is not sufficient: two families can be pointwise isospectral throughout the parameter domain without being related by any global change of basis \citep{bradlyn2017topological,vanderbilt2018berry}. In the experiments of \Cref{sec:examples}, fitting against incomplete invariant sets drives the training residual to the numerical floor while the recovered family stays at equivalence distance $\mathcal{O}(1)$ from the target class.
Classical inverse spectral theory resolves a related non-uniqueness by augmenting the spectrum with auxiliary data such as a second spectrum, norming constants, or Weyl-type functions \citep{freiling2001inverse, yurko2006inverse}. These constructions are tailored to differential operators, are not gauge invariant, and do not transfer to parameterized matrix families.

Complete invariants for this equivalence do exist: the continuous trace words~\citep{procesi1976invariant}
\begin{equation*}
    T_k(\xx_1,\ldots,\xx_k;H) := 
    \tr\bigl(H(\xx_1)H(\xx_2)\cdots H(\xx_k)\bigr)
\end{equation*}
determine the gauge class. 
The trace words are unwieldy for numerical purposes for three reasons: (i) they are multi-point and non-local, since $T_k$ couples $k$ distinct configurations, so evaluating the $k$-point words on $s$ sampled configurations costs $\mathcal{O}(s^k)$. (ii) their number is not controlled, in that no finite subset is singled out. (iii) no stability estimate is available; nothing bounds the distance between gauge classes in terms of the discrepancy between trace words, so agreement to within a tolerance carries no guarantee.



For numerical purposes, a complete invariant should be evaluable at controlled cost and admit a stability estimate. With this motivation, we make three contributions:
First, in the spirit of loop-product methods for graph-gauge problems \citep{korotyaev2017magnetic}, we construct an invariant vector field $\map(\xx; H)$, which can be thought of as a {\em pointwise} analogue of trace words. We prove that it is a complete invariant under a non-degeneracy assumption and a connectivity assumption on the interaction graph (\Cref{thm:complete_invariant_new}), which we argue are generically satisfied (\Cref{subsec:exceptional_sets}). Second, we reduce this invariant to a polynomially sized subset that remains complete. The reduced invariant records $\mathcal{O}(\dof\size^2)$ scalars per configuration, its evaluation is linear in the number of sampled configurations, and it shrinks when few pairs of states interact. Third, we prove that the reduced invariant is Lipschitz stable (\Cref{thm:full_stability}): informally, on a compact and path-connected set $K$, 
\begin{equation*}
    \min_{\QO \in \ON} \max_{\xx \in K}
    \bigl\| H_1(\xx) - \QO^\trans H_2(\xx) \QO \bigr\|
    \le C \max_{\xx \in K} \| \map(\xx; H_1) - \map(\xx; H_2) \|_\infty,
\end{equation*}
so that recovery of the gauge class is well-posed and not merely injective. We also show that the construction is sharp, in that each component of the invariant and each structural assumption is necessary.

Identifiability up to a structural ambiguity is a recurring theme in inverse problems. In the anisotropic Calder\'on problem, the conductivity tensor is determined by boundary measurements only up to a boundary-fixing diffeomorphism \citep{uhlmann2009electrical}; in state-space realization theory, a minimal linear system is determined by its input-output map only up to a change of state coordinates \citep{kalman1963mathematical}; for the magnetic Schr\"odinger operator, the magnetic potential is determined only up to a gauge transformation \citep{krupchyk2014uniqueness}. The present setting is of the same type: the admissible non-uniqueness is a global orthogonal change of basis.

\vskip 0.2cm

\noindent
{\bf Outline.}
\Cref{sec:problem_setting} formalizes gauge equivalence, introduces the derivative couplings, interaction graphs and loop products, and states the completeness theorem.
\Cref{sec:examples} first establishes sharpness: each invariant component and each hypothesis is necessary, as we show through counterexamples and optimization experiments. It then develops refinements and extensions, including the Lipschitz stability estimate, and the relaxation to removable exceptional sets.
%
\Cref{sec:conclusion} summarizes the contributions and collects open problems, and \Cref{sec:related_work} contrasts our framework with related concepts in the literature, including the continuous trace words discussed above.
\Cref{app:proof_reduction} gives the proof of \Cref{thm:full_stability}.

\vskip 0.2cm

\noindent
{\bf Notation.} 
Throughout the paper, $\size\in\Z_+$ denotes the dimension of the matrices under consideration, and $\dof\in\Z_+$ denotes the dimension of the configuration space. 
Let $\DomX\subset\R^\dof$ be a configuration domain, by which we mean a connected open subset of $\R^\dof$. 
We write $\SymHN$ for the space of real-symmetric $\size\times \size$ matrices, $\SkewHN$ for the space of real skew-symmetric $\size\times \size$ matrices, and $\ON$ for the group of real orthogonal $\size\times \size$ matrices.
For $A\in\SymHN$, we write $\lambda_1(A)\le\cdots\le\lambda_\size(A)$ for its
eigenvalues arranged in non-decreasing order; we call $A$ non-degenerate when its eigenvalues are distinct, $\lambda_1(A)<\cdots<\lambda_\size(A)$. 
A family $A(\xx), \xx \in \DomX$, is called non-degenerate if all $A(\xx)$ are non-degenerate. 
We denote by $A^\trans$ the transpose of $A$, and by $A^\dagger$ its conjugate transpose.
The symbol $\langle\cdot,\cdot\rangle$ denotes the standard Euclidean inner product on $\R^\size$, and $\partial_{x_\mu}$ the partial derivative in the configuration coordinate $x_\mu$. 
When $\dof=1$, we write the (scalar) configuration variable as $x$.

\section{Problem Setting and Main Result}
\label{sec:problem_setting}

\subsection{Gauge equivalence and invariants}
\label{sec:def:map}
Let $\DomX \subset \R^{\dof}$ be open. We consider families $H : \DomX \to \SymHN$ of real-symmetric matrices, taking $C^1\bigl(\DomX, \SymHN\bigr)$ as the default admissible class and occasionally imposing other regularity. (The complex Hermitian case under the unitary gauge will be treated in a separate work; see also \Cref{sec:conclusion}.)
In the learning context of \Cref{sec:introduction}, $\size$ is the number of basis functions of the model and $\xx \in \DomX$ the physical configuration it depends on.
Restricting the changes of basis to global orthogonal ones leads to the following equivalence relation.

\begin{definition}[Gauge equivalence]
\label{def:gauge_equivalence}
Let $H_1, H_2 \in C^1(\DomX, \SymHN)$. We say that $H_1$ and $H_2$ are \emph{gauge equivalent}, written as $H_1 \sim H_2$, if there exists 
$\QO \in \ON$ such that
\begin{equation*}
    H_2(\xx) = \QO^{\trans} H_1(\xx) \QO
\qquad \forall\, \xx \in \DomX.
\end{equation*}
\end{definition}


We study the relation through quantities that are constant on equivalence classes. We therefore consider maps 
$\map : C^1(\DomX, \SymHN) \to \cod$ into a feature space $\cod$, and ask when they are rich enough to distinguish equivalence classes.

\begin{definition}[Invariant and complete invariant]
\label{def:invariant_map}
A map $\map : C^1(\DomX, \SymHN) \to \cod$ is called an \emph{invariant} if
\begin{equation*}
\label{map:invar}
    H_1 \sim H_2 \quad\Longrightarrow\quad \map(H_1) = \map(H_2) ,
\end{equation*}
and a \emph{complete invariant} if, in addition,
\begin{equation*}
\label{map:com:invar}
    \map(H_1) = \map(H_2) \quad\Longrightarrow\quad H_1 \sim H_2 .
\end{equation*}
\end{definition}


The goal of this paper is to construct a complete invariant that is convenient for numerical purposes: pointwise, moderate computational cost and stable inversion.
The most natural starting point is the pointwise spectrum, which is incomplete. 
If two matrix families $H_1, H_2$ are isospectral, then they are related by a configuration-dependent gauge $Q(\xx)$, but not necessarily by a constant one.
For example, if 
\begin{equation}
\label{eq:spectrum_incomplete}
    \DomX = \R,  \quad H_1 = \diag(\lambda_1, \lambda_2),
    \quad H_2 = R(x)^\trans H_1\, R(x),
\end{equation}
with $R(x)$ a rotation through angle $x$ and $\lambda_1 < \lambda_2$, a straightforward calculation shows that no constant gauge can exist. The remainder of the paper introduces invariants that resolve this deficit.

\subsection{Berry connection and derivative coupling}
\label{sec:berry}
The example \eqref{eq:spectrum_incomplete} highlights what the spectrum misses: it determines $H(\xx)$ at each fixed $\xx$ up to an orthogonal frame, but cannot determine how those frames are connected by a single, configuration-independent rotation. 
To recover the gauge class we need to track how the eigenframe of $H(\xx)$ rotates as $\xx$
varies. This is encoded by the \emph{derivative coupling}.

Assume that $H$ is non-degenerate and $U \subset \DomX$ is a simply connected neighbourhood. Then there exists a $C^1$ orthonormal ordered eigenframe \citep{Kato1995}, i.e., 
\begin{equation*}
\label{eq:berry_eigenframe}
    \eframe=(u_1,\dots,u_\size)\in C^1(U,\ON),
    \qquad
    \eframe^\trans H\,\eframe = \diag\bigl(\lambda_1,\dots,\lambda_\size\bigr), 
\end{equation*}
where $\lambda_i \in C^1(U)$ and $\lambda_1 < \dots < \lambda_\size$ in $U$. 
The \emph{derivative coupling matrices} of $H$ associated with the eigenframe $\eframe$ are
\begin{equation*}
    \label{eq:berry_def}
    \calA^\mu(\xx;H) := \eframe(\xx)^\trans\,
    \partial_{x_\mu}\eframe(\xx)\in\SkewHN, \qquad \xx \in U,
    \quad \mu = 1, \dots, \dof, 
\end{equation*}
with off-diagonal entries
\begin{equation}
    \label{eq:berry_entries}
    \calA^\mu_{mn}(\xx;H)
    = \bigl\langle u_m(\xx),\,\partial_{x_\mu}u_n(\xx)\bigr\rangle
    = \frac{\bigl\langle u_m(\xx),\,\partial_{x_\mu}H(\xx)\,u_n(\xx)\bigr\rangle}{\lambda_n(\xx)-\lambda_m(\xx)},
    \qquad m\neq n.
\end{equation}
They appear as the derivative couplings in
non-adiabatic molecular dynamics \citep{yarkony1996diabolical, subotnik2016understanding}.
On the simply connected $U$ the eigenframe is $C^1$, so $\calA^\mu$ is continuous; in this regular form it is the (non-Abelian) \emph{Berry connection} \citep{vanderbilt2018berry}, and the gauge-rigidity computation below relies on this regularity.

Together with the spectrum, the derivative coupling identifies the gauge class within $U$: Let $H_1, H_2$ be isospectral and non-degenerate with ordered eigenframes $\eframe_i \in C^1(U, \ON)$ such that
    $\eframe_1^\trans H_1 \eframe_1 = \Lambda = \eframe_2^\trans H_2 \eframe_2$. 
The candidate for a constant gauge is $Q = \eframe_1 \eframe_2^\trans$.
Suppose the derivative couplings agree, writing $\calA^\mu := \calA^\mu(\cdot;H_1) = \calA^\mu(\cdot;H_2)$;
then
\begin{equation}
\label{eq:berry_rigidity_calc}
    \begin{aligned}
        \partial_{x_\mu}Q
    &=
    (\partial_{x_\mu}\eframe_1)\eframe_2^\trans + \eframe_1(\partial_{x_\mu}\eframe_2)^\trans
    \\
    &= \eframe_1\calA^\mu \eframe_2^\trans + \eframe_1(\calA^\mu)^\trans \eframe_2^\trans
    \\
    &= \eframe_1\bigl(\calA^\mu+(\calA^\mu)^\trans\bigr)\eframe_2^\trans
    = 0,
    \end{aligned}
\end{equation}
since $\calA^\mu$ is skew. Therefore $Q$ is constant and hence $H_1 = Q H_2 Q^\trans$ in $U$.


Two structural difficulties arise, together with a practical one.
First, $\calA$ is not an invariant. The eigenframe of a single $H(\xx)$ is fixed only up to right multiplication $\eframe \mapsto \eframe S$ by an element of the \emph{sign gauge}
\begin{equation*}
    \label{eq:berry_sign_group}
    \constS := \bigl\{\diag(\sigma_1,\dots,\sigma_\size):~\sigma_k\in\{\pm1\}\bigr\},
\end{equation*}
under which the coupling transforms as
\begin{equation}
    \label{eq:berry_sign_action}
    \calA^\mu \mapsto S \calA^\mu S, \qquad \text{i.e.,} \qquad
    \calA^\mu_{mn} \mapsto \sigma_m\sigma_n \calA^\mu_{mn}.
\end{equation}

Second, when $\DomX$ is not simply connected, a global $C^1$ eigenframe need not exist at all: the eigenvector lines may form non-orientable bundles, and the global-frame construction above no longer applies. This is a simplified version of the gauge ambiguities of the adiabatic representation \citep{yarkony1996diabolical}.

Compounding both, a numerical eigensolver assigns these signs arbitrarily, and need not do so consistently or continuously as $\xx$ varies, even on a simply connected domain.
This motivates the following relaxed definition, which is consistent with numerical practice. 

\begin{definition}[Pointwise derivative coupling]\label{def:ptwise_derivative_coupling}
    Let $H \in C^1(\DomX, \SymHN)$. 
    A family $\calA : \DomX \to \SkewHN^{\dof}$ is a {\em pointwise derivative coupling} of $H$ if, for each $\xx \in \DomX$, there exists an ordered $C^1$ eigenframe $\eframe$ in a neighbourhood of $\xx$ such that $\calA^\mu(\xx;H) = \eframe(\xx)^\trans \partial_{x_\mu} \eframe(\xx)$. 
\end{definition}

At each $\xx$ the possible choices of pointwise derivative couplings of $H$ differ only by a sign gauge $S(\xx)\in\constS$ (\Cref{lem:ptwise_sign_gauge_A}); since this sign may be chosen independently at every $\xx$, the representatives can be arbitrarily irregular.

%
%

Next, towards constructing true invariants from the pointwise derivative couplings we record how any two pointwise derivative couplings of a single Hamiltonian are related.

\begin{lemma}
    \label{lem:ptwise_sign_gauge_A}
    Let $H \in C^1(\DomX, \SymHN)$ be non-degenerate and let $\calA, \tilde{\calA}$ be pointwise derivative couplings of $H$.
    Then there exists $S : \DomX \to \constS$ such that $\calA^\mu(\xx;H) = S(\xx) \tilde{\calA}^\mu(\xx;H) S(\xx)$
    for all $\xx \in \DomX$ and $\mu = 1, \dots, \dof$.
\end{lemma}

\begin{proof} 
Fix $\xx \in \DomX$. 
There exist eigenframes $\eframe$ and $\tilde{\eframe}$, defined in a common neighbourhood $U\ni\xx$, such that
\begin{equation}
\label{eq:derivative_coupling_for_single_H}
\calA^\mu(\yy;H) = \eframe(\yy)^\trans\partial_{y_\mu}\eframe(\yy)
\quad \text{ and } \quad \tilde{\calA}^\mu(\yy;H) = \tilde{\eframe}(\yy)^\trans\partial_{y_\mu}\tilde{\eframe}(\yy)\qquad \forall \yy \in U.
\end{equation} 
Since $\eframe$ and $\tilde{\eframe}$ are both ordered orthonormal eigenframes of $H$ on $U$, $S \coloneqq \tilde{\eframe}^\trans\eframe$ commutes with the ordered spectrum at each $\yy \in U$. 
Since the eigenvalues of $H(\yy)$ are simple and $S(\yy) \in \ON$, it follows that $S(\yy) \in \constS$ for all $\yy \in U$.
Moreover, the continuity of both $\eframe$ and $\tilde{\eframe}$ implies that the discrete-valued map $S$ is constant over $U$.
As the two frames are related by $\eframe(\yy) = \tilde{\eframe}(\yy) S$ for $\yy \in U$, the stated result follows by inserting this relation into \eqref{eq:derivative_coupling_for_single_H}.
\end{proof}

\subsection{2-loop invariants}
\label{sec:2-loop_invariants}

As mentioned above, the derivative coupling is
not an invariant in the sense of \Cref{def:invariant_map}, since it depends on the choice of eigenframe; cf.\ \eqref{eq:berry_sign_action} and \Cref{lem:ptwise_sign_gauge_A}.
We must therefore extract from $\calA$ quantities that are insensitive to these signs. 
The simplest are the squared off-diagonals,
\begin{equation*}
\label{eq:two_loop}
    \bigl(\calA^\mu_{mn}(\xx;H)\bigr)^2 = -\,\calA^\mu_{mn}(\xx;H)\,\calA^\mu_{nm}(\xx;H),
\end{equation*}
which are clearly invariant under \eqref{eq:berry_sign_action}; we call them
\emph{$2$-loops}. 

We show that for $\size=2$ the $2$-loops already suffice.
When $\size=2$ there is a single off-diagonal entry, and its square already
exhausts the sign-covariant content of the coupling. 
In the following \Cref{prop:completeness_n2}, we restrict to a one-dimensional parameter space ($\DomX \subseteq \R$), omitting the directional index $\mu$ to write simply $\calA_{12}$, while higher-dimensional parameter spaces ($\dof > 1$) require cross-directional information, captured by the mixed loop products of \Cref{sec:graph}. 

\begin{proposition}[Completeness for $\size = 2, \dof = 1$]
\label{prop:completeness_n2}
Let $\DomX \subseteq \R$ be connected and let $H_1, H_2 \in C^1(\DomX, \Sym{2})$ be non-degenerate, isospectral families of $2 \times 2$ matrices. If
\begin{equation*}
    \bigl(\calA_{12}(\xx;H_1)\bigr)^2 = \bigl(\calA_{12}(\xx;H_2)\bigr)^2 \neq 0 \qquad \forall\, \xx \in \DomX,
\end{equation*}
then $H_1 \sim H_2$.
\end{proposition}

\begin{proof}
For $\dof = 1$, connectedness implies simple connectedness, hence there exist ordered $C^1$ eigenframes $\tilde{\eframe}^{(i)}$ and associated {\em continuous} derivative couplings $\tilde{\calA}_i(\cdot;H_i)$. 
According to \Cref{lem:ptwise_sign_gauge_A}, there exist $S_i(\xx) \in \constS$ such that $\calA(\xx;H_i) = S_i(\xx) \tilde{\calA}(\xx;H_i) S_i(\xx)$, and consequently $\tilde{\calA}_{12}(\xx;H_1)^2 = \tilde{\calA}_{12}(\xx;H_2)^2$.
Let $\tilde{a}_i(\xx) = \tilde{\calA}_{12}(\xx;H_i)$; then $\tilde{a}_1(\xx) = s(\xx) \tilde{a}_2(\xx)$ with $s(\xx) \in \{\pm1\}$.
By continuity, it follows that $s$ is constant. 
If $s = -1$, replace $\tilde{u}^{(2)}_2$ by $- \tilde{u}^{(2)}_2$ to obtain a new eigenframe with $s = 1$.
Thus, $\tilde{a}_1 = \tilde{a}_2$, that is $\tilde\calA(\cdot;H_1) = \tilde\calA(\cdot;H_2)$, and \eqref{eq:berry_rigidity_calc} implies that $H_1 \sim H_2$. 
\end{proof}
For $\dof > 1$ and $\DomX$ connected but not simply connected, additional complications arise; we postpone their discussion to the proof of the full result in \Cref{subsec:complete_invariant}.

For $\size\geq3$, by contrast, the squared couplings no longer determine the class; the
obstruction is one of parity. 
On a $3\times3$ coupling the sign gauge
\eqref{eq:berry_sign_action} acts on the three off-diagonals by
\begin{equation*}
    (\calA_{12},\calA_{23},\calA_{13})
    \;\longmapsto\;
    (\sigma_1\sigma_2\,\calA_{12},\;\sigma_2\sigma_3\,\calA_{23},\;\sigma_1\sigma_3\,\calA_{13}),
\end{equation*}
and the three sign factors satisfy $(\sigma_1\sigma_2)(\sigma_2\sigma_3)(\sigma_1\sigma_3)=+1$.
The gauge can therefore flip the sign of an \emph{even} number of entries but
never of a single one, so the triple product $\calA_{12}\calA_{23}\calA_{31}$ is
gauge invariant and is \emph{not} recoverable from the three squares
$\calA_{12}^2,\calA_{23}^2,\calA_{13}^2$.
This residual sign is unseen by the 2-loops.
\begin{example}[Cycle-parity at $\size=3$]
\label{ex:berry_3x3_counterexample}
Take $\Lambda=\diag(\lambda_1,\lambda_2,\lambda_3)$ with distinct $\lambda_i$ and
\begin{equation}
    H_\pm(x) = \exp(x\,\calA_\pm)\,\Lambda\,\exp(-x\,\calA_\pm),
    \qquad 
    \calA_\pm = {
        \scriptscriptstyle
    \begin{pmatrix} 0 & 1 & 1 \\ -1 & 0 & \pm 1 \\ - 1 & \mp 1 & 0 \end{pmatrix}
    }.
\end{equation}
Both families are isospectral with constant spectrum $\{\lambda_1,\lambda_2,\lambda_3\}$, and a direct computation gives the configuration-independent couplings $\calA(H_\pm)=\calA_\pm$. 
The two differ only in the sign of the $(2,3)$ entry, so all three $2$-loops coincide, $\calA_{mn}(H_+)^2=\calA_{mn}(H_-)^2$, while the
triple products $\calA_{12}\calA_{23}\calA_{31}$ have opposite signs. 
The families are not gauge equivalent: a constant $\QO$ with $H_-(x)=\QO^\trans H_+(x)\QO$ must, evaluated at $x=0$, satisfy $\QO^\trans\Lambda\QO=\Lambda$, forcing $\QO=\diag(\sigma_1,\sigma_2,\sigma_3)\in\constS$;
differentiating at $x=0$ then forces
$\sigma_1\sigma_2=\sigma_1\sigma_3=1$ on the $(1,2),(1,3)$ edges
but $\sigma_2\sigma_3=-1$ on the $(2,3)$ edge, a contradiction.
\end{example}

The squared couplings must therefore be supplemented by products of
coupling entries around longer cycles. 
These \emph{loop products}, and the
graph structure that organises them, are the subject of the next subsection. 

\subsection{Interaction graphs and general loop products}
\label{sec:graph}

To prepare for the generalization of 2-loop invariants, we introduce the interaction graph of a matrix and the loop products along its cycles. 
We use standard terminology and concepts from graph theory \citep{bollobas2012graph}.

Let $A\in \SymHN \cup \SkewHN$ be a general symmetric or skew-symmetric matrix. 
The \emph{graph} of $A$ is the graph $\graphA:=(\Verti,\edge(A))$ with vertices $\Verti=\{1,\dots,\size\}$ and edges $\edge(A) := \bigl\{\edgepair{m}{n} :m,n\in \Verti,\ m\neq n,\ A_{mn} \neq 0\bigr\}$. 
If $\Aset \subset \SymHN \cup \SkewHN$ is a finite family then we denote the aggregate graph or graph union by 
\begin{equation*}
    \graph\big(\Aset\big) := \Big(\Verti,\edge\big(\Aset\big)\Big), \qquad 
    \text{with} \quad \edge\big(\Aset\big) := \bigcup_{\mu=1}^\dof \edge(A^\mu). 
\end{equation*}
For a Hamiltonian family $H \in C^1(\DomX, \SymHN)$ we define its pointwise 
interaction graph by 
\begin{equation}
\label{graph:G_A_x_H}
    \graph\big(\calA(\xx;H)\big) := \graph\big( \{ \calA^\mu(\xx;H) \}_{\mu=1}^{\dof} \big),
\end{equation}
which is unambiguous with respect to the choice of local eigenframes by \Cref{lem:ptwise_sign_gauge_A}.
Next we define loops in this graph that yield gauge-invariant scalars from products of entries of $\calA$.

\begin{definition}[Loops and loop products]
\label{def:mixed_loop_product}

Let $\{A^\mu\}_{\mu=1}^\dof \subset \SkewHN$.
A \emph{$k$-loop} ($k \geq 2$) in the aggregate graph $ \graph(\{A^\mu\}_{\mu=1}^\dof)$ is a sequence
\begin{equation*}
  \bgamma=\big( (e_1, \mu_1), \dots, (e_k, \mu_k) \big)  
\end{equation*}
such that each edge $e_i = \{m_i, m_{i+1}\} \in \edge(A^{\mu_i})$ and $m_{k+1} = m_1$. 
We denote the length by $|\bgamma|:=k$.

The \emph{loop product} along a loop $\bgamma$ is 
\begin{equation*}
\Pi\big(\bgamma;\Aset\big) := A^{\mu_1}_{m_1 m_2}\, A^{\mu_2}_{m_2 m_3}\cdots A^{\mu_k}_{m_k m_1}.
\end{equation*}
We say that $\Pi(\bgamma;\Aset)$ is a \emph{pure loop product} if all $\mu_i$ coincide, and a \emph{mixed loop product} otherwise.
\end{definition}

The relevance of loop products is that they are invariant under the residual
sign gauge (\Cref{prop:loop_product_invariance}) and, as the completeness proof in \Cref{subsec:complete_invariant} shows, resolve the remaining sign ambiguities.

\begin{proposition}[Sign-gauge invariance of loop products]
\label{prop:loop_product_invariance}
Let $A^1,\dots,A^\dof\in\SymHN\cup\SkewHN$, and let $S=\diag(\sigma_1,\dots,\sigma_\size)\in\constS$ be a sign matrix. 
Define $\tilde A^\mu := S A^\mu S$; then $\graph(\tildeAset)=\graph(\Aset)$.
Furthermore, for every loop $\bgamma$ in $\graph(\Aset)$, the loop product is exactly invariant
\begin{equation*}
    \Pi\big(\bgamma;\tildeAset\big) = \Pi\big(\bgamma;\Aset\big).
\end{equation*}
That is, the loop products depend only on the sign-gauge orbit of $\Aset$. 
\end{proposition} 
\begin{proof}
Since $\sigma_m\sigma_n\in\{\pm1\}$, the map $A^\mu_{mn}\mapsto\sigma_m\sigma_n A^\mu_{mn}$ preserves the zero pattern of each $A^\mu$, so $\edge(\tilde A^\mu)=\edge(A^\mu)$ for every $\mu$ and hence $\graph(\tildeAset)=\graph(\Aset)$. 
For a $k$-loop $\bgamma=\big((e_1,\mu_1),\dots,(e_k,\mu_k)\big)$ with $e_r=\{m_r,m_{r+1}\}$
\begin{equation*}
    \Pi\big(\bgamma;\tildeAset\big)
    = \prod_{r=1}^{k}\sigma_{m_r}\sigma_{m_{r+1}}\, A^{\mu_r}_{m_r m_{r+1}}
    = \Bigl(\prod_{r=1}^{k}\sigma_{m_r}^2\Bigr)\prod_{r=1}^{k} A^{\mu_r}_{m_r m_{r+1}}
    = \Pi\big(\bgamma;\Aset\big),
\end{equation*}
where each vertex index appears exactly twice around the closed loop, so the sign factors cancel in pairs.
\end{proof}

The number of loops increases rapidly with both $\size$ and $\dof$.
In the worst case, when the aggregate graph is complete and every edge is present in all $\dof$ directions, there are $\frac{\size!}{(\size-k)!} \dof^k$ distinct $k$-loops: each of the $\frac{\size!}{(\size-k)!}$ ordered vertex cycles can be decorated with any of $\dof^k$ direction labels.
To use them in practice we extract a generating subset.
This is achieved next via a spanning tree and the corresponding fundamental cycles, a technique adopted from graph-gauge problems~\citep{korotyaev2017magnetic}.

\subsection{A complete and efficient pointwise invariant}
\label{subsec:complete_invariant}
Let $H\in C^1(\DomX,\SymHN)$ have
non-degenerate spectrum $\lambda_1(\xx)<\cdots<\lambda_\size(\xx)$ and couplings
$\calA(\xx;H)$.
The first component of our invariant is the family of spectra $( \lambda_m(\xx) )_{\xx \in \DomX, m=1,\dots,\size}$. 

The second component records, on \emph{every} vertex pair, the pure 2-loop over all directions, 
\begin{equation}
\label{eq:edge-detection}
\map^{(2\rm p)}(\xx;H) := \Bigl(\calA^\mu_{mn}(\xx;H)^2\Bigr)_{\substack{1\le m<n\le\size \\ \mu = 1, \dots, \dof}}. 
\end{equation}
In particular, $\map^{(2\rm p)}$ encodes the
zero pattern of $\calA(\xx;H)$, equivalently the edge set of the interaction
graph $\graph(\calA(\xx;H))$, at every $\xx$.
This is what makes the remaining
components well-defined as invariants: although the trees, labels, and cycles
below are read off the reference family, $\map^{(2\rm p)}$ forces two families with equal
invariants to share the same interaction graph.

The additional structural hypothesis we require next is that the aggregate interaction graph $\graph(\calA(\xx;H)) = (\Verti, \edge)$ is connected at every configuration $\xx \in \DomX$. 
Then, for each edge $e \in \edge, \xx \in \DomX$, there exists an {\em active direction label} $\ell_e(\xx) \in \{1, \dots, \dof\}$ such that $\calA^{\ell_e(\xx)}_e(\xx;H) \neq 0$. 
The collection $\ell := (\ell_e(\xx))_{e \in \edge, \xx \in \DomX}$ constitutes an admissible {\em active-direction selection}. 

An active-direction selection results in the {\em mixed 2-loop invariant} 
\begin{equation}
\label{eq:mixed-2-loop-active}
    \map^{(2\rm m)}_{\ell}(\xx;H) := \Bigl( \calA^{\ell_{e}(\xx)}_{e}(\xx;H)\, \calA^\mu_{e}(\xx;H) \Bigr)_{\substack{e \in \edge(\graph(\calA(\xx;H))) \\ \mu \in \{1,\dots,\dof\} \setminus \{\ell_e(\xx)\}}},
\end{equation}
which constitutes the third component of our complete invariant.
While the total number of mixed 2-loops scales as $\mathcal{O}(\dof^2 |\edge|)$, the active-direction selection reduces this to $\mathcal{O}(\dof|\edge|)$.

Finally, we reduce the set of general loop products. 
Since we assume that the aggregate graph is connected there exist \emph{spanning trees} $\tree(\xx) = (\Verti,\edge(\tree(\xx)))$ for each $\xx \in \DomX$; i.e., a connected acyclic subgraph that contains all vertices. 
A spanning tree is in general not unique; but for our first main result any choice suffices.
Let $e \in \edge\setminus \edge(\tree)$. 
The \emph{fundamental cycle} of $e$ with respect to $\tree$ is the unique loop obtained by adding $e$ to $\tree$. 
This results in a \emph{fundamental mixed cycle basis}
\begin{equation*}
    \begin{aligned}
    \Fset_{T, \ell}(\xx) := 
    \bigl\{ \bgamma_e : e \in \edge \setminus \edge(\tree)\bigr\}, \quad 
    \text{where}  & \quad 
    \bgamma_e := \big( (e_1, \ell_{e_1}(\xx)), \dots, (e_k, \ell_{e_k}(\xx)) \big), 
    \end{aligned}
\end{equation*}
Since $\graph$ is connected, every spanning tree $\tree$ has $|\Verti|-1$
edges. 
Hence the cardinality of the set $\Fset$ is $\beta =
|\edge|-|\edge(\tree)|=|\edge|-|\Verti|+1$ (the cyclomatic or first Betti number of $\graph$). 
The number of fundamental cycles scales as $\mathcal{O}(|\edge|)$. 
The resulting fundamental-cycle products are collected in the fourth and final component of the invariant,
\begin{equation*} 
    \map^{({\rm f})}_{\tree,\ell}(\xx;H) := 
    \Big( \Pi\big(\bgamma;\calA(\xx;H) \big) \Big)_{\bgamma \in \Fset_{T,\ell}(\xx)}. 
\end{equation*}
The fundamental cycles may have arbitrary length in general; in the generic case
where the interaction graph is complete, they reduce to triangles, and the
construction is worked out explicitly in \Cref{rem:three_loop_complete_graph}. 

Together with the spectrum these assemble into the invariant map
\begin{equation}
\label{eq:invariant_map_main}
\map_{\tree, \ell}(\xx;H) := \Bigl(\bigl(\lambda_n(\xx; H)\bigr)_{n=1}^\size, 
\map^{(2\rm p)}(\xx;H), \map^{(2\rm m)}_\ell(\xx;H), \map^{({\rm f})}_{\tree, \ell}(\xx;H)
\Bigr).
\end{equation}
Note that the map $\map_{\tree, \ell}$ can be evaluated on an arbitrary family $\tilde{H}$, which need not be the reference Hamiltonian $H$ that induced the choices of $\tree$ and $\ell$.
We also remark that the arbitrariness in the construction of $\map_{\tree, \ell}$ can be removed by making deterministic choices of $\tree(\xx)$ and $\ell_e(\xx)$; cf. \Cref{rem:deterministic_map}. 

We are now ready to state our first main result. Under mild additional assumptions and a more involved selection of $\tree, \ell$, it can be strengthened to establish {\em stable inversion}, which we will see in \Cref{thm:full_stability}. 

\begin{theorem}[Completeness]
\label{thm:complete_invariant_new}
Let $H_1 \in C^1(\DomX,\SymHN)$ be non-degenerate, with connected aggregate interaction graph $\graph\bigl(\calA(\xx;H_1)\bigr)$ for every $\xx\in\DomX$.
Let $\ell$ be an associated admissible active-direction selection and $\tree$ a spanning tree family of $H_1$. Then, a second non-degenerate family $H_2 \in C^1(\DomX,\SymHN)$ belongs to the gauge class of $H_1$, i.e., $H_1 \sim H_2$ if and only if
\begin{equation*}
    \map_{\tree, \ell}(\xx;H_1) = \map_{\tree, \ell}(\xx;H_2)
        \qquad \forall~\xx \in \DomX. 
\end{equation*}
At each $\xx$, the invariant $\map_{\tree,\ell}(\xx;\cdot)$ comprises a total of $\mathcal{O}(\dof\size^2)$ scalars.
\end{theorem}

\begin{proof}
The forward implication is immediate. 
Suppose $H_2(\xx)=\QO^\trans H_1(\xx)\QO$ for all $\xx$ with a constant $\QO\in\ON$. 
If $\eframe_1$ is a local ordered eigenframe for $H_1$, then $\eframe_2:=\QO^\trans\eframe_1$ is one for $H_2$, with identical ordered eigenvalues and, since $\QO$ is constant, identical derivative couplings $\calA^\mu(\cdot;H_2)=\calA^\mu(\cdot;H_1)$. 
Any other choice of local eigenframes alters the couplings only by simultaneous sign conjugation (\Cref{lem:ptwise_sign_gauge_A});
this leaves the two-loop and fundamental-cycle products invariant (\Cref{prop:loop_product_invariance}). 
Hence $\map_{\tree,\ell}(H_1)=\map_{\tree,\ell}(H_2)$. 

\medskip

We now prove the converse. Assume $\map_{\tree,\ell}(H_1)=\map_{\tree,\ell}(H_2)$. 
Then the ordered spectra agree: $\lambda_n(\xx;H_1)=\lambda_n(\xx;H_2)$ for all $n=1,\ldots,\size$ and $\xx\in\DomX$. 
Equality of $\map^{(2\rm p)}$ ensures that $\calA^\mu_{mn}(\xx;H_1)^2 = \calA^\mu_{mn}(\xx;H_2)^2$ for every $\xx$, every pair $\{m,n\}$ and every direction $\mu \in \{1, \dots, \dof\}$.
Hence the two families have a common interaction graph, which we denote $\graph(\xx)=(V,\edge(\xx)):=\graph(\calA(\xx;H_1))=\graph(\calA(\xx;H_2))$.

{\it Step 1: Reconstruct $\calA$ up to a sign: }
We first reconstruct the sign relation at a fixed point $\xx$. 
Fix an edge $e \in \edge(\xx)$ and let $\ell = \ell_e(\xx)$. 
Equality of the pure and mixed two-loop yields, for every $\mu$,
\begin{equation*}
    \calA^\ell_{e}(\xx;H_1)\calA^\mu_{e}(\xx;H_1) = \calA^\ell_{e}(\xx;H_2)\calA^\mu_{e}(\xx;H_2).
\end{equation*}
Taking $\mu=\ell$ gives $(\calA^\ell_{e}(\xx;H_1))^2=(\calA^\ell_{e}(\xx;H_2))^2$, ensuring the existence of a sign $s_e(\xx)\in\{\pm1\}$ such that $\calA^\ell_{e}(\xx;H_2)=s_e(\xx)\calA^\ell_{e}(\xx;H_1)$. Substituting this into the mixed identities and cancelling the non-zero factor $\calA^\ell_{e}(\xx;H_1)$, we obtain $\calA^\mu_{e}(\xx;H_2)=s_e(\xx)\calA^\mu_{e}(\xx;H_1)$ for all $\mu=1,\ldots,\dof$. Thus, on each edge $e$, the two families differ by a single, direction-independent edge sign
\begin{equation*}
    \calA^\mu_{e}(\xx;H_2) = s_e(\xx)\calA^\mu_{e}(\xx;H_1), \qquad \mu=1,\ldots,\dof, \quad e \in \edge(\xx).
\end{equation*}

{\it Step 2: Trivial cycle holonomy: }
We now use the fundamental-cycle products to show that these edge signs have trivial cycle holonomy.
Let $e$ be a non-tree edge and let $\bgamma_e$ be its corresponding fundamental cycle. 
Because all labels on $\bgamma_e$ are canonically selected to be non-zero, the corresponding cycle product is non-zero.
The equality of the fundamental-cycle products,
\begin{equation*}
    \Pi\big(\bgamma_e;\calA(\xx;H_2)\big) = \Pi\big(\bgamma_e;\calA(\xx;H_1)\big),
\end{equation*}
expands, upon applying the per-edge sign relation $\calA^\mu_e(\xx;H_2)=s_e(\xx)\calA^\mu_e(\xx;H_1)$ edge by edge, to
\begin{equation*}
    \Big(\prod_{f\in\bgamma_e} s_f(\xx)\Big)\,\Pi\big(\bgamma_e;\calA(\xx;H_1)\big)
    = \Pi\big(\bgamma_e;\calA(\xx;H_1)\big),
\end{equation*}
and since the cycle product is non-zero, the coupling factor cancels to leave
\begin{equation*}
    \prod_{f\in\bgamma_e} s_f(\xx)=1 .
\end{equation*}

{\it Step 3: Factorize edge signs into vertex signs: }
Because the fundamental cycles span a complete basis of the cycle space over $\mathbb F_2$ \citep{bollobas2012graph}, it follows that $\prod_{f\in \mathfrak C}s_f(\xx)=1$ for every cycle $\mathfrak C$ in $\graph(\xx)$.
Equivalently, the signed graph $(\graph(\xx),s(\xx))$ is balanced \citep{harary1953notion}. 
Consequently, there exist vertex signs $\sigma_1(\xx),\ldots,\sigma_\size(\xx)\in\{\pm1\}$ such that
\begin{equation*}
    s_{\{m,n\}}(\xx)=\sigma_m(\xx)\sigma_n(\xx) \qquad \forall\,\{m,n\}\in \edge(\xx).
\end{equation*}
This is constructively verified by choosing a root vertex $v_0$, setting $\sigma_{v_0}=1$, and defining $\sigma_v$ as the product of edge signs along any path from $v_0$ to $v$; the trivial cycle holonomy guarantees path-independence.
Thus, at every $\xx$, there exists a diagonal sign matrix $S(\xx)=\operatorname{diag}(\sigma_1(\xx),\ldots,\sigma_\size(\xx))\in\constS$ satisfying
\begin{equation*}
    \calA^\mu(\xx;H_2) = S(\xx)\calA^\mu(\xx;H_1)S(\xx), \qquad \mu=1,\ldots,\dof .
\end{equation*}
(Off the interaction graph both sides vanish, by the equality of the two graphs.)

{\it Step 4: Locally constant signs: } 
Let $U \subset \DomX$ be an open and simply connected neighbourhood of $\xx_0 \in \DomX$. Let $\eframe_1, \eframe_2$ be $C^1$-eigenframes for $H_1, H_2$ in $U$, and let $\tilde{\calA}(\cdot;H_1), \tilde{\calA}(\cdot;H_2) \in C(U, \SkewHN)$ be the associated derivative couplings (now genuine Berry connections) in $U$. Let $\tilde\tree := \tree(\xx_0)$ and $\tilde\ell := \ell(\xx_0)$; by continuity of $\tilde{\calA}(\cdot;H_1)$, these remain admissible for all $\xx \in U$, provided $U$ is sufficiently small.

\Cref{lem:ptwise_sign_gauge_A} states that there exist $S_i(\xx) \in \constS, \xx \in U$, such that 
\begin{equation*}
    \calA^\mu(\xx;H_i) = S_i(\xx) \tilde{\calA}^\mu(\xx;H_i) S_i(\xx), 
    \qquad \mu = 1, \dots, \dof. 
\end{equation*}
Combining this identity with Step 3, we obtain $\tilde{S}(\xx) = S_1(\xx) S(\xx) S_2(\xx) \in \constS$ such that 
\begin{equation*}
    \tilde{\calA}^\mu(\xx;H_2) = \tilde{S}(\xx)\tilde{\calA}^\mu(\xx;H_1)\tilde{S}(\xx), \qquad \mu=1,\ldots,\dof.
\end{equation*}
We write $\tilde{S}(\xx) = \diag(\tilde\sigma_1(\xx), \dots, \tilde\sigma_\size(\xx))$;
for an edge $e = \{m, n\}$ of $\tilde\tree$ we obtain
\begin{equation*}
    \tilde\sigma_m(\xx) \tilde\sigma_n(\xx) = 
    \frac{
        \tilde{\calA}^{\tilde{\ell}_e}_e(\xx;H_2)
    }{
        \tilde{\calA}^{\tilde{\ell}_e}_e(\xx;H_1)
    }
    \in \{ \pm 1 \}, \qquad \xx \in U.
\end{equation*}
Since the right-hand side is continuous and discrete-valued on the connected set $U$, it is constant, and hence so is the left-hand side $\tilde\sigma_m(\xx)\tilde\sigma_n(\xx)$. The vertex signs are products of these now-constant tree-edge signs along tree paths, so the individual $\tilde\sigma_i(\xx)$ are constant as well.

{\it Step 5: Locally constant gauge: } Proceeding from Step 4, let $Q_U(\xx):=\tilde{\eframe}_1(\xx) \tilde{S} \tilde{\eframe}_2(\xx)^\trans$.
Because the eigenvalues match and $\tilde{S}$ commutes with the diagonal eigenvalue matrix, we obtain $H_2(\xx)=Q_U(\xx)^\trans H_1(\xx) Q_U(\xx)$ for all $\xx\in U$.
Furthermore, calculating the derivative yields
\begin{equation}
\label{eq:real_rigidity_dQ}
    \partial_{x_\mu}Q_U = \tilde{\eframe}_1\tilde{\calA}_1^\mu \tilde{S} \tilde{\eframe}_2^\trans - \tilde{\eframe}_1\tilde{S} \tilde{\calA}_2^\mu\tilde{\eframe}_2^\trans= \tilde{\eframe}_1\bigl(\tilde{\calA}_1^\mu \tilde{S} -\tilde{S} \tilde{\calA}_2^\mu\bigr)\tilde{\eframe}_2^\trans = 0,
\end{equation}
since $\tilde{\calA}_2^\mu=\tilde{S} \tilde{\calA}_1^\mu \tilde{S} $. Consequently, $Q_U$ is constant in $U$. This yields a constant orthogonal gauge on every sufficiently small ball.

{\it Step 6: Globally constant gauge: } It remains to glue the local gauges together into a global one.
Let $B$ and $B'$ be two balls with non-empty intersection, and let $Q_B,Q_{B'}$ be the corresponding constant gauges. 
For any $\xx\in B\cap B'$, we have $Q_B^\trans H_1(\xx)Q_B = Q_{B'}^\trans H_1(\xx)Q_{B'}$.
Defining $R:=Q_BQ_{B'}^\trans$, it follows that $R^\trans H_1(\xx)R=H_1(\xx)$. 
Given $H_1(\xx)$ is non-degenerate, $R$ expressed in an ordered eigenbasis of $H_1(\xx)$ must be a diagonal sign matrix, i.e., $\widehat R := \eframe_1(\xx)^\trans R\,\eframe_1(\xx) = {\rm diag}(r_1, \dots, r_\size) \in \constS$.

Differentiating the gauge relations implies $R^\trans\partial_{x_\mu}H_1(\xx)R = \partial_{x_\mu}H_1(\xx)$, and \eqref{eq:berry_entries} then implies
%
\begin{equation*}
    (\lambda_n(\xx)-\lambda_m(\xx))\calA^\mu_{mn}(\xx;H_1) = (\lambda_n(\xx)-\lambda_m(\xx))r_m r_n\calA^\mu_{mn}(\xx;H_1).
\end{equation*}
Because the interaction graph at $\xx$ is connected, this forces all diagonal entries of $\widehat R$ to coincide. Thus $R=\pm I$, enforcing $Q_{B'}=\pm Q_B$.
Because the sign $\pm$ is transparent to the conjugation, we may replace $Q_{B'}$ with $\pm Q_{B'}$ to obtain a single constant gauge on $B \cup B'$. 

To conclude, let $B_1, B_2, B_3, \dots$ be a countable cover of $\DomX$ with constant gauges $Q_{i}$ in $B_i$. Since $\DomX$ is connected we can choose $B_i$ to have non-empty intersection with $\bigcup_{j<i} B_j$. Taking the foregoing argument as an induction step, we can construct signs such that $Q_i = \pm Q_j$ and obtain a global constant gauge $\QO\in\ON$ such that
\begin{equation*}
    H_2(\xx)=\QO^\trans H_1(\xx)\QO \qquad \forall\,\xx\in\DomX.
\end{equation*}
Therefore, $H_1\sim H_2$.
\end{proof}

\begin{remark}
The gluing argument from Step 6 in the foregoing proof requires $\DomX$ to be connected, but not simply-connected. 
Topological obstructions arise when one glues objects with sign or phase ambiguities, which is exactly why a global continuous eigenframe does not generally exist. 
Here, however, we are not gluing the gauge matrices $Q_B$, but their conjugation action $H \mapsto Q_B^\trans H\, Q_B$. 
Because the overlap relation is $Q_{B'} = \pm Q_B$, the residual sign $\pm I$ cancels under conjugation. 
Therefore, the conjugation action is unambiguously defined on every overlap. 
With no sign ambiguity to resolve, this locally constant action extends uniquely to the entire connected domain, regardless of its topology.
\end{remark}

\begin{remark}[Deterministic construction yields a single complete invariant]
\label{rem:deterministic_map}
As constructed, $\map_{\tree,\ell}$ depends on two arbitrary choices at each
configuration $\xx$: a spanning tree for the aggregate interaction graph $\tree(\xx)$ and active direction label $\ell(\xx)$. 
This suits the application of our results for numerical inversion problems ideally, 
but does not provide a single map on
$C^1(\DomX,\SymHN)$ in the sense of \Cref{def:invariant_map}, but a family of
maps indexed by $(\tree,\ell)$. 
We can remove this arbitrariness by making the choices canonical and reading them off each input $H$. 

Assuming a fixed vertex ordering $1<\cdots<\size$, we order the unoriented edges
$\{i,j\}$ lexicographically by the pair $(\min\{i,j\},\max\{i,j\})$. 
For a connected graph $\graph(\calA(\xx;H))$, we define $\tree(\xx)$ as the
lexicographically first spanning tree, comparing spanning trees via their
ordered edge lists. 
For each edge $e=\{m,n\}\in \edge(\graph(\calA(\xx;H)))$, the
active label can then be chosen as 
\begin{equation*}
    \ell_{mn}(\xx)
    :=
    \min\bigl\{
        \mu\in\{1,\ldots,\dof\} \mid \calA^\mu_{mn}(\xx;H)\neq 0
    \bigr\}.
\end{equation*}
These rules depend strictly on the zero pattern of the derivative couplings.
Because a change in the local eigenframe merely conjugates $\calA^\mu$ by a
diagonal sign matrix, all zero patterns are preserved; consequently
$\tree(\xx)$ and $\ell(\xx)$ are independent of the pointwise sign
gauge. 
Although $\xx\mapsto\tree(\xx)$ and $\xx\mapsto\ell(\xx)$ need
not be continuous, this poses no difficulty: they serve solely as pointwise
indexing rules for the reduced loop products, and the completeness argument
imposes no regularity on them. 
With these canonical rules in hand, we let each family $H$ induce its own 
``canonical'' map 
\begin{equation*}
    \map_\ast(H)(\xx) := \Big(
        \tree(\xx), \ell(\xx), 
        \map_{\tree,\ell}(\xx;H)
        \Big).
\end{equation*}
It then follows from \Cref{thm:complete_invariant_new} that $\map_\ast$ is a complete invariant in the sense of \Cref{def:invariant_map}. That is, $H_1 \sim H_2$ if and only if $\map_\ast(H_1)(\xx) = \map_\ast(H_2)(\xx)$ for all $\xx \in \DomX$.
\end{remark}

\begin{remark}[The complete graph: fundamental cycles are triangles]
\label{rem:three_loop_complete_graph}

We illustrate the fundamental cycle in   \Cref{subsec:complete_invariant} on the complete graph $K_\size$. 
The lexicographic rule of
\Cref{rem:deterministic_map} compares spanning trees by their ordered edge
lists, and the lexicographically smallest edges are 
$\{1,2\}, \{1,3\}, \ldots, \{1,\size\}$. These edges form
a spanning tree, namely the star $\tree_*$ centred at vertex $1$, with $\edge_{\tree_*}=\{\{1,j\} : 2\le j\le\size\}$. Each non-tree edge $\{m,n\}$
($2\le m<n\le\size$) closes the fundamental triangle $\bgamma_{mn}=(1,m,n,1)$ 
(\Cref{fig:complete_graph_three_loops}). 
The reduced invariant then consists of
the spectrum, the pure and mixed two-loop products, and the
$\binom{\size-1}{2}$ triangle products
$\Pi\bigl((1,m,n,1);\calA(\xx;H)\bigr)$, with no loop of length greater than
three.

This case is the relevant one in practice: for randomly sampled families, all coupling entries are non-zero almost
surely, so the interaction graph $\graph(\calA(\xx;H))$ is almost surely
complete, so the fundamental cycles can be taken to be triangles.
\begin{figure}[htbp]
\centering
\begin{tikzpicture}[scale = 0.9,
  vertex/.style={circle, draw=black, thick, fill=white,
                 minimum size=8pt, inner sep=0pt},
  treedge/.style={red, ultra thick},
  chord/.style={blue, dashed, thick},
  cyc/.style={green!60!black, line width=3pt, opacity=0.35},
  panellbl/.style={font=\small\bfseries},
  framed/.style={draw=black!30, rounded corners=2pt, inner sep=8pt}
] 
\node[framed, minimum width=5cm, minimum height=4cm] (boxC) at (0,0) {};
\begin{scope}[shift={(0,0)}]
    \foreach \i in {1,...,5} {
    \node[vertex] (v\i) at ({90-(\i-1)*72}:1.35) {};
    \node[font=\scriptsize] at ({90-(\i-1)*72}:1.7) {$\i$};
  }
  \node[vertex, fill=red!15] at (v1) {};

  \draw[chord] (v2) -- (v3);
  \draw[chord] (v2) -- (v4);
  \draw[chord] (v2) -- (v5);
  \draw[chord] (v3) -- (v4);
  \draw[chord] (v3) -- (v5);
  \draw[chord] (v4) -- (v5);

  \draw[cyc] (v1) -- (v2);
  \draw[cyc] (v2) -- (v3);
  \draw[cyc] (v3) -- (v1);

  \draw[treedge] (v1) -- (v2);
  \draw[treedge] (v1) -- (v3);
  \draw[treedge] (v1) -- (v4);
  \draw[treedge] (v1) -- (v5);
\end{scope}
\node[panellbl, anchor=south west] at ([xshift=3pt,yshift=3pt]boxC.south west) {$K_5$};
\end{tikzpicture}
\vskip 0.2cm
\caption{
Complete graph $K_5$ with the canonical star spanning tree rooted at vertex 1. Tree edges are shown in solid red, non-tree edges in dashed blue, and one fundamental triangle $(1,2,3,1)$ is highlighted in green. For any $K_\size$, every non-tree edge $\{m,n\}$ (where $2\leq m<n\leq\size$) generates a corresponding triangle $(1,m,n,1)$.
}
\label{fig:complete_graph_three_loops}
\end{figure}
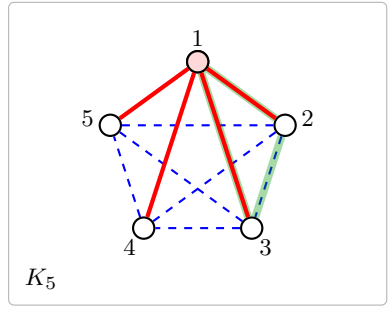
\end{remark}

\section{Examples and Refinements}
\label{sec:examples}
This section refines the theory of \Cref{sec:problem_setting} in two stages. We first establish its \emph{sharpness}: each component of the invariant map and each structural hypothesis of \Cref{thm:complete_invariant_new} is necessary, shown through explicit counterexamples and accompanying numerical experiments. We then collect a series of \emph{refinements and extensions} of the basic theory:
\begin{list}{}{%
  \setlength{\leftmargin}{2em}
  \setlength{\labelsep}{0.5em}
  \setlength{\itemsep}{0.4em}%
  \let\makelabel\descriptionlabel}
  \item[\Cref{sec:necessity} (\emph{\nameref{sec:necessity}})] demonstrates the strict necessity of each invariant component, showing that fundamental cycles are required for $\size\geq 3$, and mixed 2-loops are required for $\dof\geq 2$. 
  \item[\Cref{section:connectedness} (\emph{\nameref{section:connectedness}})] illustrates domain connectedness issues via $C^{\infty}$ examples that are locally but not globally gauge equivalent when $\DomX$ disconnects, a defect we show is removable in the analytic category.
  \item[\Cref{subsec:exceptional_sets} (\emph{\nameref{subsec:exceptional_sets}})] formalizes the conditions under which exceptional sets (where assumptions of the main results are violated) do not affect the unique inversion.
  \item[\Cref{subsec:well_posedness} (\emph{\nameref{subsec:well_posedness}})] establishes a stability estimate demonstrating that on compact subdomains the inversion from a well-chosen invariant map is well-posed.
\end{list}

Some examples are theoretical; others are numerical, framed as a
gauge-class recovery optimization problem: a reference family $H_\star$ is
generated, and a trainable model $H_\theta$ is fit to its invariants by minimizing the mean-squared discrepancy
on a finite training set $\DomX_{\rm train}\subset\DomX$. We report two
quantities. The \emph{training RMSE} is the square root of the training
loss, 
\begin{equation*}
\rmse(\theta):=
    \Bigl(
        \tfrac{1}{|\DomX_{\rm train}|}
        \sum_{\xx\in\DomX_{\rm train}}
        \bigl\| \map_{\rm sub}(\xx;H_\theta)
              - \map_{\rm sub}(\xx;H_\star) \bigr\|^2
    \Bigr)^{1/2},
\end{equation*}
where $\map_{\rm sub}$ stacks the components of the selected invariant
subset. We use the nested subsets $\setspec$ (spectrum only),
$\setpp$ (adding pure 2-loops on all vertex pairs),
$\settwo$ (adding mixed 2-loops), and $\setfc$ (adding
fundamental-cycle products). The \emph{equivalence distance}, 
\begin{equation*}
    \dist(H_\theta, H_\star)
    :=
    \min_{\QO\in\ON}
    \max_{\xx\in\DomX_{\rm test}}
    \bigl\| H_\theta(\xx) - \QO^\trans H_\star(\xx)\,\QO \bigr\|_F ,
\end{equation*}
measures recovery of the gauge class on a test set $\DomX_{\rm test}$. 
A small $\rmse$ with $\dist=\mathcal{O}(1)$ signals that the recorded invariants do
not pin down the gauge class: the model reproduces the training data yet sits in
the wrong class. This happens either because the chosen invariant set is
incomplete (the necessity experiments of \Cref{sec:necessity}) or because a
hypothesis of \Cref{thm:complete_invariant_new} fails---for instance a
disconnected domain (\Cref{section:connectedness}). Both quantities reaching the
optimizer floor confirm exact recovery.

All experiments use PyTorch with L-BFGS optimizer. Unless noted otherwise, $100$ restarts with i.i.d.\ standard normal initialization, at most $3000$ iterations each,
$20$ training and $20$ test configurations; restarts with terminal loss above \num{1e-6} are discarded as spurious local minima. 

\subsection{Minimality of the invariants selection}
\label{sec:necessity}

The completeness theorem (\Cref{thm:complete_invariant_new}) shows that the full invariant map \emph{suffices} to identify a family up to gauge; here we show it is also \emph{minimal}, in that no single component can be dropped without losing completeness. 

The spectral and two-loop obstructions were already constructed as counterexamples in \Cref{sec:problem_setting}: the family in \eqref{eq:spectrum_incomplete} is everywhere isospectral to a constant diagonal family, but its eigenframe rotates with the configuration variable and cannot be fixed by a single constant gauge; by \Cref{ex:berry_3x3_counterexample}, two families can share the same spectrum and 2-loops yet differ by a cycle sign that only a fundamental-cycle product detects. 
We next recast \Cref{ex:berry_3x3_counterexample} as a gauge-recovery optimization problem, showing that in practice, two-loop data are insufficient for $\size\ge 3$. 

\begin{example}[Fundamental cycles are needed for $\size\geq 3$]
\label{ex:cycle_parity_obstruction}

Let $\DomX = (-1, 1)$. 
We generate a reference family $H_{\star}(x) = R_{\star}(x)^\trans \Lambda_{\star}(x) R_{\star}(x)$, where the spectrum $\Lambda_{\star}(x)$ has positive gaps ($\lambda_1=p_1(\xx)$,
$\lambda_{n+1}=\lambda_n+\exp(p_{n+1}(\xx))$) 
and the
eigenframe $R_{\star}(x) = \exp(x\Omega_{\star})$ uses a fixed generator 
$\Omega_{\star} \in \SkewHN$ whose upper-triangular entries are
drawn i.i.d.\ from the standard normal distribution. 
We fit it with a parameterized model $H_\theta(x) = R_\theta(x)^\trans \Lambda_\theta(x) R_\theta(x)$ of identical structure. 
Both the eigenvalue gap parameters in $\Lambda_\theta(x)$ and the upper-triangular entries of $\Omega_\theta \in \SkewHN$ are fully trainable. The results, reported in \Cref{tab:cycle_parity_training_obstruction}, show that,
on $\setspec$ or $\setpp$ data, the optimizer drives the $\rmse$ to the
numerical floor while the equivalence distance stays at $\mathcal{O}(1)$. 
Adding the fundamental-cycle products removes 
these traps, confirming that the cycle products are
needed to resolve the residual sign ambiguity for $\size \ge 3$.

\begin{table}[htbp]
\centering
\small
\caption{Optimization-based verification of the cycle-parity obstruction (see
\Cref{ex:cycle_parity_obstruction}). Here $\dof=1$, so the mixed 2-loops are
empty and $\settwo$ coincides with $\setpp$; only the latter is shown. The
incomplete sets $\setspec$ and $\setpp$ reach the training $\rmse$ floor but
fail to recover the gauge class, retaining an $\mathcal{O}(1)$ equivalence
distance; adding the fundamental-cycle products ($\setfc$) drives both
quantities to the numerical floor.}
\label{tab:cycle_parity_training_obstruction}
\begin{tabular}{ccccc}
\toprule
 & & \multicolumn{3}{c}{Invariant set} \\
\cmidrule(lr){3-5}
$N$ & Metric & \setspec & \setpp & \setfc \\
\midrule
3 & $\rmse$ & \num{1.5e-10} & \num{2.8e-09} & \num{1.8e-09} \\
  & $\dist$ & \num{3.4e+00} & \num{1.3e+00} & \num{1.4e-08} \\
\addlinespace
4 & $\rmse$ & \num{1.4e-09} & \num{2.0e-09} & \num{1.9e-09} \\
  & $\dist$ & \num{4.5e+00} & \num{3.4e+00} & \num{1.8e-08} \\
\addlinespace
5 & $\rmse$ & \num{1.1e-09} & \num{1.4e-09} & \num{2.8e-09} \\
  & $\dist$ & \num{3.8e+00} & \num{2.7e+00} & \num{2.7e-08} \\
\addlinespace
6 & $\rmse$ & \num{2.7e-09} & \num{8.7e-10} & \num{4.9e-09} \\
  & $\dist$ & \num{9.1e+00} & \num{8.2e+00} & \num{2.0e-08} \\
\bottomrule
\end{tabular}
\end{table}
\end{example}

We next show that for $\dof\ge2$ the \emph{pure} 2-loops cannot align coupling signs across parameter directions, so the \emph{mixed} 2-loops are required to separate the equivalence classes. To isolate this cross-directional obstruction, we consider an example with $\size = 2$. In this case, the interaction graph contains no cycles. Consequently, any failure to recover the gauge must stem from signs left unaligned across directions. 

\begin{example}[Mixed 2-loops are needed for $\dof>1$]
\label{ex:pure-2-loops-are-insufficient}
Let $\Lambda=\operatorname{diag}(\lambda_1,\lambda_2)$ with $\lambda_1\neq\lambda_2$ and $\Omega\in\Skew{2}$. 
Define
$H_\pm(\xx) = R_\pm(\xx)^\trans \Lambda R_\pm(\xx)$ with
$R_\pm(\xx) = \exp(x_1\Omega \pm x_2\Omega)$. 
Both families share the constant
spectrum $\Lambda$ and have derivative couplings
$\calA_+^1 = \calA_-^1 = \Omega$, $\calA_+^2 = \Omega$, $\calA_-^2 = -\Omega$.
Their pure 2-loops match in every direction, so the spectrum together with
pure 2-loops cannot tell them apart. 
Yet a constant gauge $Q$ with
$H_- = Q^\trans H_+ Q$ must commute with $\Lambda$, hence
$Q = \operatorname{diag}(\sigma_1,\sigma_2)$, and $\calA_-^\mu = Q\calA_+^\mu Q$
demands $\sigma_1\sigma_2=1$ for $\mu=1$ but $\sigma_1\sigma_2=-1$ for $\mu=2$,
a contradiction. 

For the numerical test, we keep the same $\size=2$ structure and use a single
base generator with direction-dependent scalar weights. More precisely, we set $R_\star(\xx)=\exp(\Big(\sum_{\mu=1}^{\dof} x_\mu c^\star_\mu\Big)\Omega_\star)$,  
where the scalar coefficients $c^\star_\mu$ are sampled randomly and
$\Omega_\star$ is a nonzero skew-symmetric generator. 
The model $H_\theta$ uses the same ansatz, $R_\theta(\xx) = \exp(\Big(\sum_{\mu=1}^{\dof} x_\mu c^\theta_\mu\Big)\Omega_\theta)$, 
with the gap parameters of $\Lambda_\theta$, the scalar coefficients
$c^\theta_\mu$, and the generator $\Omega_\theta\in\Skew{2}$
fully trainable. 
\Cref{tab:pure_two_loop_M_training_obstruction} shows that training on
$\setspec$ or $\setpp$ drives the $\rmse$ to the numerical floor while the
equivalence distance stays large; adding the mixed 2-loops
($\settwo$) brings the distance to the floor as well, confirming that mixed
2-loops align signs across directions once $\dof \ge 2$. 

\begin{table}[htbp]
\centering
\small
\renewcommand{\arraystretch}{1.1}
\caption{Optimization-based verification of the cross-directional sign
obstruction for $\dof>1$ and $\size=2$ (see
\Cref{ex:pure-2-loops-are-insufficient}). The incomplete sets $\setspec$ and
$\setpp$ reach the training $\rmse$ floor but retain an $\mathcal{O}(1)$
equivalence distance, whereas $\settwo$ resolves the cross-directional sign and
drives both quantities to the numerical floor. Since $\size=2$, $\setfc$ coincides with $\settwo$; it is therefore not
shown separately.} 
\label{tab:pure_two_loop_M_training_obstruction}
\begin{tabular}{ccccc}
\toprule
 & & \multicolumn{3}{c}{Invariant set} \\
\cmidrule(lr){3-5}
$\dof$ & Metric & \setspec & \setpp & \settwo \\
\midrule
2 & $\rmse$ & \num{8.2e-10} & \num{4.3e-09} & \num{2.3e-09} \\
  & $\dist$ & \num{1.7e+00} & \num{1.7e+00} & \num{3.8e-09} \\
\addlinespace
3 & $\rmse$ & \num{1.2e-09} & \num{3.5e-09} & \num{9.9e-09} \\
  & $\dist$ & \num{3.4e+00} & \num{7.5e-01} & \num{8.3e-08} \\
\addlinespace
4 & $\rmse$ & \num{2.4e-09} & \num{5.1e-09} & \num{7.3e-09} \\
  & $\dist$ & \num{7.1e-01} & \num{7.0e-01} & \num{1.8e-08} \\
\bottomrule
\end{tabular}
\end{table}
\end{example}

Finally, we show that the pure 2-loops must be recorded even on non-edges. 

\begin{example}[Non-edge pure 2-loops are necessary]
\label{ex:edge_detection_needed}
Let $\size=3$, $\dof=1$, $\varepsilon\neq0$, and let $\Lambda=\diag(\lambda_1,\lambda_2,\lambda_3)$ have distinct entries.
Consider the skew generators
\begin{equation*}
\Omega_1= \begin{pmatrix}0&1&0\\-1&0&1\\0&-1&0\end{pmatrix},
\qquad
\Omega_2=\begin{pmatrix}0&1&\varepsilon\\-1&0&1\\-\varepsilon&-1&0\end{pmatrix},
\end{equation*}
and set $H_i(x)=\exp(x\Omega_i)\Lambda\exp(-x\Omega_i)$ for $i=1,2$.
Since $\Lambda$ is constant and each $\Omega_i$ commutes with $\exp(x\Omega_i)$, the derivative coupling is constant, $\calA(\cdot;H_i)=\Omega_i$.
Thus $\graph(\calA(\cdot;H_1))$ is the tree with edges $\{1,2\}$ and $\{2, 3\}$, while $\graph(\calA(\cdot;H_2))$ carries the additional chord $\{1,3\}$.
The two families share the same constant spectrum $\Lambda$, and their generators agree on the tree edges, so the pure 2-loops on $\{1,2\}$ and $\{2,3\}$ coincide.
With a single direction there are no mixed 2-loops, and the tree $\graph(\calA(\cdot;H_1))$ has no fundamental cycles; hence every component of $\map_{\tree,\ell}$ agrees for $H_1$ and $H_2$ \emph{except} the non-edge pure 2-loop on $\{1,3\}$, which is $0$ for $H_1$ and $\varepsilon^2$ for $H_2$.
Yet $H_1\not\sim H_2$, because the interaction graph is a gauge invariant and the two graphs differ by the edge $\{1,3\}$.

We recast this as a gauge-recovery optimization. Fixing the reference $H_1$, we
fit a model $H_\theta(x)=\exp(x\Omega_\theta)\,\Lambda\,\exp(-x\Omega_\theta)$ whose
generator $\Omega_\theta$ has trainable entries on $\{1,2\}$, $\{2,3\}$, and the
chord $\{1,3\}$, with chord magnitude $|\varepsilon_\theta|$. Without the non-edge
2-loops the chord is invisible to the loss: the $\rmse$ reaches the numerical floor 
while leaving $\dist=\mathcal{O}(1)$. Adding the non-edge 2-loop on $\{1,3\}$ forces
$\varepsilon_\theta^2\to0$ and drives both quantities to the floor
(\Cref{tab:nonedge_pure_twoloop_training}), confirming that $\map^{(2\rm p)}$ on
non-edges is indispensable.

\begin{table}[htbp]
\centering
\small
\caption{Training verification that non-edge pure 2-loops are needed (see
\Cref{ex:edge_detection_needed}). Without the non-edge two-loop, the model reaches small $\rmse$ in the
wrong gauge class ($\dist=\mathcal{O}(1)$); adding it recovers the class. 
The residual $\dist$ scales as \emph{square root} of the $\rmse$: the only invariant detecting the chord in this experiment is the non-edge pure 2-loop, which is quadratic in the chord amplitude and vanishes at the reference. In \Cref{subsec:well_posedness} we introduce extended cycle products which are \emph{linear} in the chord entry against a tree-path factor bounded below. With those included in the loss, \Cref{thm:full_stability} establishes a linear rate.}
\label{tab:nonedge_pure_twoloop_training}
\begin{tabular}{lccc}
\toprule
Invariant set & $\rmse$ & $\dist$ & $|\varepsilon_\theta|$ \\
\midrule
$\setpp$, edges only        & \num{2.1e-11} & \num{3.7e+00} & \num{2.7e+00} \\
$\setpp$, all vertex pairs  & \num{9.6e-09} & \num{4.2e-04} & \num{1.3e-04} \\
\bottomrule
\end{tabular}
\end{table}
\end{example}

\subsection{Necessity of connectedness}
\label{section:connectedness}

The connectedness of the admissible domain is equally critical. 
Global identifiability relies on both a non-degenerate family and a persistently connected interaction graph. 
If either fails, the domain can split into disjoint regions on which incompatible local gauges are glued across the gap.

\begin{example}[Pointwise connectedness is necessary]
\label{ex:pointwise_connectedness}
Let $\DomX=(-1,1)$, let $a\neq b$, and define the flat function $\rho(x)= e^{-1/x^2}$ for $x\neq 0$ and $\rho(0)=0$. 
Define the families
\begin{equation*}
H_1(x)= \begin{pmatrix} a & \rho(x)\\ \rho(x) & b \end{pmatrix}, \quad H_2(x)= \begin{pmatrix} a & \operatorname{sgn}(x)\rho(x)\\ \operatorname{sgn}(x)\rho(x) & b \end{pmatrix}.
\end{equation*}
Both families are $C^\infty$ and share the same simple eigenvalues for all $x\in(-1,1)$. 
For $x>0$, we have $H_2(x)=H_1(x)$; for $x<0$, we have $H_2(x)=D^\trans H_1(x)D$ with $D=\operatorname{diag}(1,-1)$. 
Their spectral and loop-product invariant therefore agree. 
At $x=0$, however, the coupling entry and all its derivatives vanish, so the graph disconnects. 
No single constant orthogonal matrix $Q$ satisfies $H_2(x)=Q^\trans H_1(x)Q$ globally on $(-1,1)$: for $x>0$ the identity $H_2=H_1$ forces $Q=\pm I$, while for $x<0$ the relation requires $Q=\pm D$, a contradiction. 
Thus \Cref{thm:complete_invariant_new} fails without pointwise connectedness, even for $C^\infty$ families. 
   
We defer the numerical verification of this trap to \Cref{ex:crossing_connected_A}, where it is reported together with the spectrum-crossing obstruction.
\end{example}

\begin{example}[Spectrum crossing with connected interaction graph]
\label{ex:crossing_connected_A}
Let $\DomX=(-1,1)$. 
Define $D = \operatorname{diag}(1,-1)$, $S = \begin{psmallmatrix} 0 & 1 \\ 1 & 0 \end{psmallmatrix}$, and use the same flat function $\rho(x) = e^{-1/x^2}$ with $\rho(0)=0$. 
Consider the $C^\infty$ families 
\begin{equation*}
  H_1(x) = \begin{pmatrix} x & \rho(x) \\ \rho(x) & -x \end{pmatrix}, \quad 
H_2(x) = \begin{pmatrix} x & \operatorname{sgn}(x)\rho(x) \\ \operatorname{sgn}(x)\rho(x) & -x \end{pmatrix}  
\end{equation*}
Both families share the exact eigenvalues $\lambda_\pm(x) = \pm\sqrt{x^2+\rho(x)^2}$, which are simple on the good set $\DomX_0 = (-1,0) \cup (0,1)$ with a degenerate crossing at $x=0$. 
The graph stays connected on $\DomX_0$, yet the crossing separates it into two components. 
As in \Cref{ex:pointwise_connectedness}, this permits incompatible local gauges ($I$ and $D$), so no global constant gauge exists.

We again frame gauge recovery as an optimization problem. In both cases we set $H_\star=H_1$ and fit it with $H_\theta$ of the same structure, where the sign of the off-diagonal entry is a separate trainable variable $s_\pm\in\{\pm1\}$ on each branch ($x<0$ and $x>0$). 
As reported in \Cref{tab:domain_separation_training}, the optimizer drives the $\rmse$ to zero in every case but, the domains being separated, glues incompatible local gauges, leaving the equivalence distance away from zero. 
This confirms that without pointwise connectedness or across a separating crossing, the complete invariant cannot determine the global gauge class.

\begin{table}[htbp]
\centering
\small
\renewcommand{\arraystretch}{1.15}
\caption{Domain-separation obstructions (\Cref{ex:pointwise_connectedness,ex:crossing_connected_A}). Trained on the complete invariant, the optimizer drives the $\rmse$ to zero but leaves the equivalence distance $\approx 0.93$: the gauge between disconnected components is not fixed.}
\label{tab:domain_separation_training}
\begin{tabular}{lcc}
\toprule
Target reference & Metric & \setpp \\
\midrule
Graph disconnect ($a=0, b=1$) & $\rmse$ & \num{4.2e-13} \\
 & $\dist$ & \num{9.3e-01} \\
\addlinespace
Graph disconnect ($a=0, b=2$) & $\rmse$ & \num{1.1e-13} \\
 & $\dist$ & \num{9.3e-01} \\
\addlinespace
Graph disconnect ($a=-1, b=1$) & $\rmse$ & \num{1.3e-11} \\
 & $\dist$ & \num{9.3e-01} \\
\addlinespace
Spectrum crossing & $\rmse$ & \num{5.9e-11} \\
 & $\dist$ & \num{9.3e-01} \\
\bottomrule
\end{tabular}
\end{table}

\end{example}

\begin{remark}[Analyticity removes the obstruction]
\label{rem:analyticity_removes_gluing}
The obstruction relies on merely $C^\infty$ regularity, which lets the two local gauges be chosen independently across the disconnection; it does not survive in the analytic category. 
If $H_1,H_2$ are real-analytic and $H_2=Q_-^\trans H_1 Q_-$ on a nonempty open subinterval, the identity theorem extends this across the exceptional point to the whole connected domain.
The constructions above therefore rely on non-analytic flat functions; analytic families are immune, as the next example demonstrates. 
\end{remark}

\begin{example}[Disconnectedness for polynomial family]
\label{ex:spline_vs_poly_connectedness}

For $\size \ge 2$, let $\Lambda_\size = \operatorname{diag}(1, 2, \ldots, \size)$. Let $S_\size$ be tridiagonal with $(S_\size)_{i,i+1}=(S_\size)_{i+1,i}=1$, and let $r = x_1 + \cdots + x_\dof$. 
We define the reference family as $H_1^{(\size,\dof)}(\xx) = \Lambda_\size + r^2 S_\size$. 
At the hyperplane $\{r=0\}$---a codimension-one set that separates $\DomX$---the derivatives $\partial_{x_\mu}H_1^{(\size,\dof)} = 2r S_\size$ all vanish, so every coupling entry vanishes and the interaction graph disconnects. 
For any $r \neq 0$, the graph remains connected.
The alternating sign $P_\size=\diag(1,-1,\ldots)$,
which satisfies $P_\size^\trans S_\size P_\size=-S_\size$, is the local gauge
that the two sides of $\{r=0\}$ could in principle glue inconsistently; the
experiment tests whether the analytic (polynomial) structure rules this out.

We fit the reference on a finite sample of the good set $\{r\neq0\}$ by a
degree-$\le2$ polynomial model of the same ansatz space,
\begin{equation*}
    H_\theta(\xx) = \sum_{|\bm k|\le 2} A_{\bm k}\,\xx^{\bm k},
    \qquad A_{\bm k}\in\SymHN ,
\end{equation*}
where every coefficient matrix $A_{\bm k}$ is fully trainable and initialized
with i.i.d.\ standard-normal entries (symmetrized). 
The model is therefore not
constrained to the reference form $\Lambda_\size+r^2S_\size$: it may realize any
real-symmetric polynomial family of total degree at most two. 

As shown in \Cref{tab:polynomial_removable_connectedness_defect}, adding the
pure 2-loops ($\setpp$) collapses $\dist$ by four to nine orders of magnitude
relative to $\setspec$ for every $(\size,\dof)$; the mixed 2-loops ($\settwo$)
are reported for completeness, though it adds nothing here. Thus, although the interaction graph disconnects on the separating set $\{r=0\}$, the defect is removable in the analytic category: the model is polynomial, hence analytic, so the identity theorem (\Cref{rem:analyticity_removes_gluing}) forces agreement across $\{r=0\}$. This contrasts with the non-analytic $C^\infty$ constructions of \Cref{ex:pointwise_connectedness,ex:crossing_connected_A}, where the same separation leaves $\dist > 0$.

\begin{table}[htbp]
\centering
\small
\renewcommand{\arraystretch}{1.12}
\caption{Disconnecting reference $H_1^{(\size,\dof)}=\Lambda_\size+r^2S_\size$
fitted by a degree-$\le2$ polynomial model (see
\Cref{ex:spline_vs_poly_connectedness}). Adding the 2-loop invariants ($\setpp$,
$\settwo$) collapses the equivalence distance by four to nine orders of
magnitude relative to $\setspec$, confirming that the disconnection at $r=0$ is
removable in the analytic category.} 
\label{tab:polynomial_removable_connectedness_defect}
\begin{tabular}{cc cccc}
\toprule
 & & & \multicolumn{3}{c}{Invariant set} \\
\cmidrule(lr){4-6}
$\dof$ & $\size$ & Metric & \setspec & \setpp & \settwo \\
\midrule
1 & 2 & $\rmse$ & \num{6.1e-10} & \num{1.7e-09} & \num{1.7e-09} \\
  &   & $\dist$ & \num{2.1e+00} & \num{4.8e-09} & \num{4.8e-09} \\
\addlinespace
1 & 3 & $\rmse$ & \num{7.4e-10} & \num{1.8e-09} & \num{1.8e-09} \\
  &   & $\dist$ & \num{2.6e+00} & \num{6.2e-09} & \num{6.2e-09} \\
\addlinespace
1 & 4 & $\rmse$ & \num{5.5e-05} & \num{1.7e-08} & \num{1.7e-08} \\
  &   & $\dist$ & \num{2.6e+00} & \num{5.5e-07} & \num{5.5e-07} \\
\midrule
2 & 2 & $\rmse$ & \num{3.0e-10} & \num{2.9e-09} & \num{3.3e-09} \\
  &   & $\dist$ & \num{3.9e+00} & \num{8.8e-09} & \num{8.8e-09} \\
\addlinespace
2 & 3 & $\rmse$ & \num{6.2e-05} & \num{2.0e-09} & \num{6.2e-09} \\
  &   & $\dist$ & \num{4.5e+00} & \num{1.9e-08} & \num{6.4e-08} \\
\addlinespace
2 & 4 & $\rmse$ & \num{8.1e-05} & \num{1.6e-08} & \num{3.4e-07} \\
  &   & $\dist$ & \num{4.6e+00} & \num{6.3e-07} & \num{1.6e-05} \\
\midrule
3 & 2 & $\rmse$ & \num{9.5e-10} & \num{7.1e-09} & \num{7.9e-09} \\
  &   & $\dist$ & \num{4.9e+00} & \num{2.9e-08} & \num{2.6e-08} \\
\addlinespace
3 & 3 & $\rmse$ & \num{1.5e-06} & \num{9.8e-09} & \num{1.8e-08} \\
  &   & $\dist$ & \num{7.0e+00} & \num{1.5e-07} & \num{2.1e-07} \\
\addlinespace
3 & 4 & $\rmse$ & \num{2.8e-09} & \num{2.3e-08} & \num{1.5e-05} \\
  &   & $\dist$ & \num{1.0e+01} & \num{1.1e-06} & \num{9.4e-04} \\
\bottomrule
\end{tabular}
\end{table}
\end{example}

In summary, for general smooth families both hypotheses of \Cref{thm:complete_invariant_new}---non-degenerate and a connected interaction graph---are necessary: breaking either can disconnect the domain and open a gauge-gluing trap. 
The trap can be avoided in two complementary ways. 
If the family is real-analytic, the identity theorem propagates a local gauge across the exceptional set even when that set separates the domain, as the polynomial example above shows. 
If the family is only smooth, removability instead requires that the good set remain path-connected and dense; we formalize this condition in \Cref{subsec:exceptional_sets}.

\subsection{The exceptional set}
\label{subsec:exceptional_sets}

\Cref{thm:complete_invariant_new} assumes a non-degenerate spectrum and a connected interaction graph everywhere. 
In practice, however, systems often exhibit isolated degeneracies or points where the graph disconnects. 
We ask whether such points destroy identifiability or can be resolved by continuity.

Let $\mathcal E\subset\DomX$ be an ``exceptional set'' where those assumptions fail and let $\DomX_0 := \DomX\setminus\mathcal E$. 
The exceptional set is harmless whenever the theorem's hypotheses hold on $\DomX_0$ and $\DomX_0$ remains path-connected and dense: the local gauges then patch and extend across $\mathcal E$. 

\begin{corollary}[Removable exceptional sets]
\label{cor:removable_exceptional_sets}
Let $\DomX\subset\mathbb R^\dof$ be open and connected, and let $\mathcal E\subset\DomX$ be closed. 
Assume the complement $\DomX_0 := \DomX\setminus \mathcal E$ is dense in $\DomX$ and path-connected. 
Let $H_1,H_2\in C^1(\DomX,\SymHN)$. If $H_1$ and $H_2$ satisfy the hypotheses of \Cref{thm:complete_invariant_new} on the reduced domain $\DomX_0$ with identical invariants $\map_{\tree, \ell}(\xx;H_1) = \map_{\tree, \ell}(\xx;H_2)$ for all $\xx \in \DomX_0$, then there exists a global constant orthogonal matrix $\QO\in\ON$ such that
\begin{equation*}
    H_2(\xx)=\QO^\trans H_1(\xx)\QO, \qquad \forall\,\xx\in\DomX.
\end{equation*}
\end{corollary}

\begin{proof}
Applying \Cref{thm:complete_invariant_new} directly to the connected domain $\DomX_0$ yields a constant orthogonal matrix $\QO\in\ON$ such that $H_2(\xx)=\QO^\trans H_1(\xx)\QO$ for all $\xx\in\DomX_0$. Since $\DomX_0$ is dense in $\DomX$ and the families $H_1, H_2$ are continuous, this equality uniquely extends to the closure, holding for all $\xx\in\DomX$.
\end{proof}

\begin{remark}[Exceptional sets are generically codimension-two]
\label{rem:codim_two_exceptional_sets}
What matters is not the exceptional points themselves but whether removing them disconnects the domain: a set of codimension two or higher leaves the complement dense and path-connected, so gauges continue around it, whereas a codimension-one set can separate it.

For spectrum degeneracies this is favourable. Near a simple two-level crossing, subtracting the mean of the two eigenvalues puts the effective $2\times2$ block in the trace-free form
$\begin{psmallmatrix} \alpha(\xx) & \beta(\xx) \\ \beta(\xx) & -\alpha(\xx) \end{psmallmatrix}$,
with gap $2\sqrt{\alpha(\xx)^2+\beta(\xx)^2}$; the crossing imposes the two conditions $\alpha=\beta=0$, already codimension two. More generally, let $\mathbb{E}_{k_1,\ldots,k_L}$ denote the real-symmetric matrices with $L\le\size$ distinct eigenvalues of multiplicities $k_1,\ldots,k_L\in\Z_{\ge1}$; its codimension in $\SymHN$ is
\begin{equation}\label{eq:codim_of_degenerates}
    \operatorname{codim}(\mathbb{E}_{k_1,\ldots,k_L})
    = \frac{1}{2}\sum_{i=1}^L (k_i+2)(k_i-1)
\end{equation}
\citep{keller2008multiple,vonneumann1929verhalten,arnold1971matrices}. Hence no generic spectral degeneracy is codimension-one, and the simple crossing $\mathbb{E}_{2,1,\ldots,1}$ already realizes the minimal codimension, two.

It should be possible to make these formal ideas precise using transversality \citep{golubitsky1973stable}: For a smooth family $H:\DomX\to\SymHN$, the degeneracy locus is the preimage $\mathcal E=H^{-1}(\mathbb{E}_{k_1,\ldots,k_L})$; when $H$ is transverse to $\mathbb{E}_{k_1,\ldots,k_L}$ its codimension is preserved, $\operatorname{codim}(\mathcal E)=\operatorname{codim}(\mathbb{E}_{k_1,\ldots,k_L})\ge2$. Transversality holds for generic $H$, so a codimension-one degeneracy is non-transverse and unstable: an arbitrarily small perturbation of $H$ removes it. For $\dof\ge2$ a generic crossing therefore does not separate the domain and is removable.

The codimension-one obstructions of \Cref{section:connectedness} are exactly such non-generic, fine-tuned configurations. Forcing a crossing into a single parameter (\Cref{ex:crossing_connected_A}) makes $\{x=0\}$ an isolated separating point; engineering every coupling to share a common factor (\Cref{ex:spline_vs_poly_connectedness}) disconnects the interaction graph on the hyperplane $\{r=0\}$---a graph disconnection rather than a spectral degeneracy, but a separating codimension-one set all the same. Either way the components may adopt incompatible gauges, unless the family is analytic, when the gauge extends across the separation regardless (\Cref{rem:analyticity_removes_gluing}).
\end{remark}




In summary, the non-degeneracy and connectedness hypotheses of \Cref{thm:complete_invariant_new} are needed only across exceptional sets that separate the domain, which is non-generic, and even then only for families that are not analytic.

\subsection{Well-posedness on compact subsets}
\label{subsec:well_posedness}

While the completeness theorem establishes exact identifiability, practical applications benefit from a stability result. Such a result requires replacing the qualitative assumptions of \Cref{thm:complete_invariant_new} with quantitative ones. We next explain the resulting changes to the invariant construction and assumptions.

First, we restrict inversion to a compact set $K\subset\DomX$. We will assume that it is rectifiably path-connected with finite intrinsic diameter. That is, there exists $L_K < \infty$ such that any two points in $K$ can be joined by a rectifiable path in $K$ of length at most $L_K$.

Let $H_1\in C^1(\DomX,\SymHN)$ be a non-degenerate reference family and assume, as before, that the pointwise aggregate interaction graph $\graph(\calA(\xx;H_1))$ is connected for every $\xx\in K$. By compactness of $K$, the pointwise assumptions yield three uniformity parameters on $K$:
First, we have a positive spectral gap
\begin{equation} \label{eq:stab:gap}
    g_K := \inf_{\xx\in K} \min_{m\neq n} |\lambda_m(\xx;H_1)-\lambda_n(\xx;H_1)| > 0.
\end{equation}
Second, we have a finite upper bound
\begin{equation} \label{eq:stab:upperbnd}
    B_K := \|H_1\|_{C^0(K)} + \max_{\mu=1,\ldots,\dof} \|\calA^\mu(\cdot;H_1)\|_{C^0(K)} < \infty.
\end{equation}
Third, on each edge of the chosen spanning tree family $\tree(\xx) = (\Verti, \edge(\tree(\xx)))$ the chosen active direction $\ell_e(\xx)$ carries a connection that is uniformly bounded below on $K$: we require a positive constant $\kappa_{K}(H_1)$ such that
\begin{equation} \label{eq:kappa_dir_lower_bound}
    \bigl|\calA^{\ell_e(\xx)}_{e}(\xx;H_1)\bigr| \ge \kappa_{K}(H_1) > 0, \qquad \xx\in K,\ \ e\in\edge(\tree(\xx)).
\end{equation}
While \eqref{eq:stab:gap} and \eqref{eq:stab:upperbnd} are always satisfied, the third requirement \eqref{eq:kappa_dir_lower_bound} requires a careful construction of the spanning tree family and active directions; we explain in \Cref{rem:kappa_construction} how this can be guaranteed. 

The stability analysis uses a variant of the complete invariant of \Cref{subsec:complete_invariant}, differing from it in the fundamental-cycle products. The spectral, pure 2-loops, and mixed 2-loops are as in \Cref{subsec:complete_invariant}. Here, the fundamental-cycle products attach to each non-tree vertex pair $\{p,q\}\notin\edge(\tree(\xx))$, whether or not it is an edge of $\graph(\calA(\xx;H))$, the tree path $\{f_1,\ldots,f_k\}$ joining $p$ to $q$, and to each direction $\mu$ the labelled cycle in which the tree-path edges carry their active labels and the chord carries direction $\mu$, $\bgamma_{\{p,q\}}^{\mu} := \bigl((f_1,\ell_{f_1}(\xx)),\ \ldots,\ (f_{k},\ell_{f_{k}}(\xx)),\ (\{p,q\},\mu)\bigr)$, and records the corresponding cycle products
\begin{equation*}
    \map^{(\mathrm f)}_{\tree,\ell}(\xx;H) := \Bigl(\,\Pi\bigl(\bgamma_{\{p,q\}}^{\mu};\calA(\xx;H)\bigr)\,\Bigr)_{\substack{\{p,q\}\notin\edge(\tree(\xx))\\ \mu=1,\ldots,\dof}} .
\end{equation*}
If $\{p,q\}$ is absent from the aggregate graph of $H$, then all these entries vanish and the corresponding product is zero; recording it nonetheless is what forces a competitor's coupling to vanish on non-edges (see \Cref{lem:stab_local_reconstruction}). We denote the resulting invariant by $\map_{\tree,\ell}$.

To measure recovery of the gauge class we use the equivalence distance on $K$, 
\begin{equation*}
    \dist_K(H_1,H_2) := \inf_{\QO\in\ON}\sup_{\xx\in K}\|H_2(\xx)-\QO^{\trans}H_1(\xx)\QO\|_F, 
\end{equation*}
and to measure agreement of the invariants we use the discrepancy 
\begin{equation*}
    \Delta_K(H_1,H_2) := \max_{\xx\in K}\|\map_{\tree,\ell}(\xx;H_1)-\map_{\tree,\ell}(\xx;H_2)\|_\infty. 
\end{equation*}
$\Delta_K$ is the theoretical version of the training $\rmse$ from \Cref{sec:examples}. Both measure the same error in closely related ways: while $\rmse$ measures the root mean squared error over sampled training data $\DomX_{\rm train}\subset K$, $\Delta_K$ takes the maximum error over the entire set $K$. 

\begin{theorem}[Stability from the reduced loop discrepancy]
\label{thm:full_stability}
Let $\DomX, H_1$ satisfy the assumptions of \Cref{thm:complete_invariant_new} and let $K\subset\DomX$ be compact and rectifiably path-connected with path-length bound $L_K$. Let $\tree(\xx), \ell(\xx)$ be spanning trees and active direction labels such that \eqref{eq:kappa_dir_lower_bound} is satisfied. Then, there exist $C,\delta_0 > 0$ such that, if $H_2\in C^1(\DomX,\SymHN)$ and $\Delta_K(H_1,H_2) \le \delta_0$, then
\begin{equation}
    \dist_K(H_1,H_2) \le C\,\Delta_K(H_1,H_2).
\end{equation}
 The constants $C, \delta_0$ depend on $H_1, H_2$ only through $g_K, B_K, \kappa_K, L_K,\size,\dof$.
\end{theorem}

We give the proof in \Cref{app:proof_reduction}.

\begin{remark}[The smallness threshold]
\label{rem:smallness_threshold}
Some of what the smallness of $\Delta_K \leq \delta_0$ ensures is elementary: it transfers the uniform spectral gap and the uniform upper bounds from $H_1$ to the competitor, so that $H_2$ is also non-degenerate on $K$ and admits ordered $C^1$ eigenframes with bounded derivative couplings. Beyond this, the assumption achieves two further things.

(i) It forces a uniform lower bound on the active couplings of $H_2$. Since the invariant records only the squares $(\calA_e^{\ell_e})^2$, a discrepancy below $\kappa_K^2$ leaves the active tree entries of $H_2$ bounded away from zero — the tree-activation assumption~(ii) of \Cref{lem:stab_local_reconstruction} — so that the sign of the ratio of the two active entries is well defined. This could in principle still be imposed as a ``generic'' structural assumption on $H_2$, but the smallness threshold guarantees it.

(ii) More significantly, the reconstruction error $C\Delta_K$ must satisfy $C\Delta_K<\kappa_K/2$, so that the local sign gauges cannot flip relative to one another on overlaps (\Cref{lem:stab_overlap_compatibility}). This condition compares the error with the activation modulus rather than transferring a modulus of $H_1$, and it cannot be traded for structural assumptions on $H_2$: the residual gauge group $\constS$ is discrete, so the correct branch is selected only once the discrepancy lies below the scale separating the branches. 
\end{remark}

\begin{remark}[Maximizing the connectivity modulus]
\label{rem:kappa_construction}
To realize the condition \eqref{eq:kappa_dir_lower_bound} we can employ a magnitude-based selection spanning tree. 
Let $\mathfrak \tree_\size$ denote the set of spanning trees on the
vertex set $\{1,\ldots,\size\}$, and define
\begin{equation*}
    \kappa(H_1)
    :=
    \inf_{\xx\in K}
    \max_{\tree\in\mathfrak \tree_\size}
    \min_{e \in \edge(\tree)}
    \left( \max_{\mu = 1, \dots, \dof} |\calA^\mu_{e}(\xx;H_1)| \right)
\end{equation*}
For each $\xx\in K$ the graph $\graph(\calA(\xx;H_1))$ is connected, so it admits
a spanning tree, and on each tree edge some direction carries a non-zero
coupling; hence the inner minimum over $e\in\edge(\tree)$ is positive for the
maximizing tree. Since the resulting function of $\xx$ is continuous and $K$ is
compact, $\kappa(H_1)>0$. 
Note that this provides a
magnitude-based way to choose a spanning tree, different from the lexicographic
selection in \Cref{rem:deterministic_map}: using $\max_\tree$ instead of a fixed
combinatorial rule makes the inner quantity continuous in $\xx$, and hence
uniformly positive on $K$.

The same construction selects, on each tree edge, a single active direction of
controlled strength. 
For $\xx\in K$, let $\tree(\xx)$ attain the inner maximum in
$\kappa(H_1)$. 
For each tree edge $e=\{i,j\}\in\edge(\tree(\xx))$, define
\begin{equation*}
    \ell_e(\xx)
    :=
    \min\Bigl\{
        \mu\in\{1,\ldots,\dof\}:
        |\calA^\mu_{ij}(\xx;H_1)|
        =
        \max_{1\le\nu\le\dof}|\calA^\nu_{ij}(\xx;H_1)|
    \Bigr\},
\end{equation*}
i.e. the lexicographically first direction realizing the largest coupling
magnitude on $e$. 
By the elementary norm equivalence
$\max_\mu|\calA^\mu_{ij}|\ge \dof^{-1/2}\bigl(\sum_\mu|\calA^\mu_{ij}|^2\bigr)^{1/2}$
on every tree edge of $\tree(\xx)$,
\begin{equation*}
    \inf_{\xx\in K}
    \min_{\{i,j\}\in\edge(\tree(\xx))}
    \bigl|\calA^{\ell_{ij}(\xx)}_{ij}(\xx;H_1)\bigr|
    \ge
    \frac{\kappa(H_1)}{\sqrt{\dof}}
    >
    0 ,
\end{equation*}
so \eqref{eq:kappa_dir_lower_bound} holds with
$\kappa_{K}(H_1)=\kappa(H_1)/\sqrt{\dof}$. 
This is the quantitative analogue, in the parameter direction, of the active-direction selection used to define the reduced invariant in \Cref{subsec:complete_invariant}.
\end{remark}

\begin{example}[The cycle sector cannot be dropped]
\label{eg:full_loop_and_separation}
The 2-loop sector already fixes the \emph{magnitude} of every coupling entry, since the pure 2-loop records $(\calA^\mu_{mn})^2$; the fundamental-cycle products then add only the residual sign.
Since these signs take only finitely many values, one might hope that, in the stable regime ($\Delta_K$ small), continuity and compactness suffice to pin them down from the spectral and 2-loop data alone, i.e.\ that $\dist_K(H_1,H_2)\le C\bigl(\Delta_K^{\rm spec}+\Delta_K^{(2)}\bigr)$.
We exhibit a counterexample family, which builds on \Cref{ex:edge_detection_needed}.

Let $\size=3$, $\dof=1$, $K=[-T,T]$, and let $\Lambda=\diag(\lambda_1,\lambda_2,\lambda_3)$ have distinct entries with minimal gap $g>0$. Set
\begin{equation*}
    \Omega_0=\begin{pmatrix}0&1&0\\-1&0&1\\0&-1&0\end{pmatrix},
    \qquad
    \Omega_\varepsilon=\begin{pmatrix}0&1&\varepsilon\\-1&0&1\\-\varepsilon&-1&0\end{pmatrix},
\end{equation*} 
and define the \emph{fixed} reference family and the parameterized competitors
\begin{equation*}
    H_1(t) = \exp(t\Omega_0)\,\Lambda\,\exp(-t\Omega_0),
    \qquad
    H_2^\varepsilon(t) = \exp(t\Omega_\varepsilon)\,\Lambda\,\exp(-t\Omega_\varepsilon),
    \qquad \varepsilon\in(0,1).
\end{equation*}
Since each $\Omega$ commutes with $\exp(t\Omega)$, the frames $\exp(t\Omega)$ are ordered eigenframes and the connections are constant, $\calA(\cdot;H_1)\equiv\Omega_0$ and $\calA(\cdot;H_2^\varepsilon)\equiv\Omega_\varepsilon$. The interaction graph $\graph(\calA(\cdot;H_1))$ is the tree with edges $\{1,2\}$ and $\{2,3\}$. Taking $\tree$ to be this tree and $\ell_e\equiv 1$ gives $\kappa_K=1$, while $g_K=g$, $B_K=\|\Lambda\|_F+\|\Omega_0\|$, and $L_K=2T$. All four moduli are independent of $\varepsilon$, so $H_1$ satisfies the hypotheses of \Cref{thm:full_stability}.

Evaluating the invariants exactly, the spectra agree and, with $\dof=1$, there are no mixed 2-loops; the only discrepancies occur on the non-edge $\{1,3\}$:
\begin{equation*}
\Delta_K^{\rm spec}=\Delta_K^{(2\rm m)}=0,
\qquad
\Delta_K^{(2\rm p)} = \varepsilon^2,
\qquad
\Delta_K^{(\mathrm f)} = \varepsilon .
\end{equation*}
Hence the total discrepancy is $\Delta_K(H_1,H_2^\varepsilon)=\varepsilon$, and a direct calculation shows that the true equivalence distance also scales linearly, $\dist_K(H_1,H_2^\varepsilon) = \Theta(\varepsilon)$.

Consequently, while $H_1$ remains fixed and $\epsilon \to 0$, we have $\Delta_K(H_1,H_2^\varepsilon)\to0$ as well as 
\begin{equation*}
    \frac{\dist_K(H_1,H_2^\varepsilon)}
         {\Delta_K^{\rm spec}+\Delta_K^{(2\rm p)}+\Delta_K^{(2\rm m)}}
    \;\asymp\; \frac{\varepsilon}{\varepsilon^2}
    \;\longrightarrow\;\infty .
\end{equation*}
Thus no estimate of the form $\dist_K\le C\bigl(\Delta_K^{\rm spec}+\Delta_K^{(2\rm p)}+\Delta_K^{(2\rm m)}\bigr)$ can hold on any neighbourhood of $H_1$. This is a structural effect: $H_2$ carries an extra coupling---the chord $\{1,3\}$---absent from the reference $H_1$, and the two-loop sector is blind to this new chord at linear order, detecting it only quadratically. This is why $\map^{(\mathrm f)}_{\tree,\ell}$ must be recorded on \emph{all} non-tree vertex pairs; it is the analytic counterpart of the square-root convergence rate observed numerically in \Cref{tab:nonedge_pure_twoloop_training}.
\end{example}

\begin{example}[Well-posedness: numerical illustration]
\label{ex:wellposedness_loss_eps}

Let $K=[-1,1]^\dof$ and consider dense systems with dimensions $\size\in\{3,4,5\}$ and $\dof\in\{1,2\}$. The reference family is parameterized as
\begin{equation*}
    H_\star(\xx)
    =
    \exp\Big(\textstyle\sum_{\mu=1}^\dof x_\mu\Omega_\star^\mu\Big)\,
    \Lambda_\star(\xx)\,
    \exp\!\Big(-\textstyle\sum_{\mu=1}^\dof x_\mu\Omega_\star^\mu\Big),
\end{equation*}
where each skew-symmetric generator $\Omega_\star^\mu\in\SkewHN$ is constructed with upper-triangular entries drawn independently from a standard normal distribution. The spectrum $\Lambda_\star(\xx)=\diag(\lambda_1,\ldots,\lambda_N)(\xx)$ is generated to ensure strictly positive eigenvalue gaps, with its polynomial coefficients also drawn from a standard normal distribution. We fit it with the trainable model $H_\theta$ adopting the identical analytic ansatz. The upper-triangular entries of its generators $\Omega_\theta^\mu\in\SkewHN$ and its spectral polynomial coefficients are fully trainable, initialized randomly from a standard normal distribution. 

\Cref{fig:three_loop_generic} exhibits a clear $\rmse$--error relationship
consistent with \Cref{thm:full_stability}: under the complete $\setfc$ set,
small $\rmse$ is accompanied by small equivalence distance, whereas the
incomplete $\settwo$ set produces restarts that reach the $\rmse$ floor while
retaining an $\mathcal{O}(1)$ distance.

\begin{figure}[htbp]
\centering
\includegraphics[width=\textwidth]{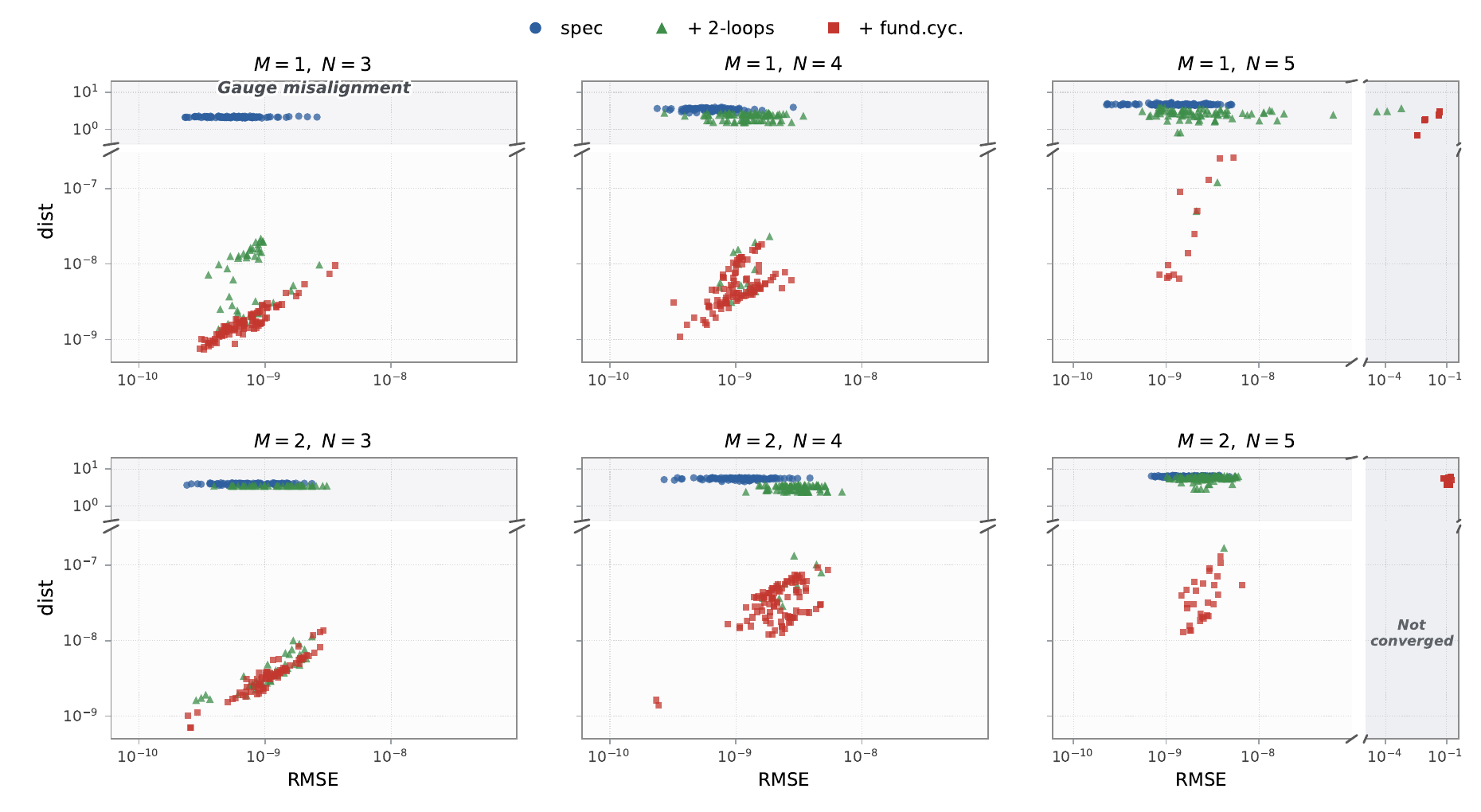}
\caption{Restart-level $\rmse$ versus equivalence distance $\dist$ (see
\Cref{ex:wellposedness_loss_eps}). Each point represents a single random
restart. Under the incomplete $\settwo$ invariant set, a band of restarts
attains the $\rmse$ floor at $\dist=\mathcal{O}(1)$, corresponding to spurious
sign branches. The complete $\setfc$ invariant set ensures that a vanishing
$\rmse$ guarantees a vanishing $\dist$.} 
\label{fig:three_loop_generic}
\end{figure}
\end{example}

\section{Conclusion}
\label{sec:conclusion}
In this paper we addressed an inverse problem: which computationally practical
observations of a parameterized matrix family determine it up to a global change
of basis. 
Our main result is that augmenting the pointwise spectrum with loop products of
the derivative couplings forms a complete invariant for the orthogonal gauge
(\Cref{thm:complete_invariant_new}), under non-degeneracy and connectivity
assumptions, and that the inversion is stable (\Cref{thm:full_stability}).
The invariant is computationally efficient in that it records only $\mathcal{O}(\dof\size^2)$ scalars per sampled configuration.

%

Three assumptions limit the scope of these results.
The first is that the model and the reference belong to the same admissible class. 
In practice a learned model $H_\theta$ and the data-generating family $H_\star$ may
differ in structure or even in dimension: the reference may be an
$\size_\star\times\size_\star$ family while the model uses $\size<\size_\star$
basis functions. 
The target is then an effective low-dimensional surrogate, and the loop-product
invariant already provides gauge-invariant features and a quotient distance for
comparing such surrogates without fixing a basis
(\Cref{subsec:relation_reduced_basis}); which gauge-invariant content survives a
projection, and how $\dist_K$ behaves under it, remains to be made precise. 

The second is that $\size$ and $\dof$ are implicitly assumed to be moderate. 
We have not studied the limit of large matrix dimension, and it is not the regime
this paper targets: the motivating application is dimension reduction, where
$\size$ is moderate by construction. 
%
%
The meaningful large-system limit is not $\size\to\infty$ with the family
unstructured, but one in which the dimension ``per particle'' stays fixed while
$\size$ and $\dof$ both grow with the number of particles, as for the
size-extensive bases discussed in \Cref{sec:introduction}. 
Formulating such a limit within the present framework, in theory and in practice
and in sufficient generality, is the key open problem we leave. 

The third is that we restricted this work to real symmetric families. The complex Hermitian case under the unitary gauge requires a genuine extension: the residual eigenframe freedom is the continuous group $\mathsf U(1)^\size$ instead of the discrete sign group $\constS$, so the first-order couplings no longer determine the gauge and second-order information must enter the invariant. We treat that case in a separate work.

Finally, the present results are theoretical: they settle what is identifiable from basis-independent observations and how stably. It remains to turn our proposed invariants into a practical loss for Hamiltonian learning (the application that motivates this work), understanding sample complexity, and comparing the resulting methods against established ones on realistic data.


\vskip 0.2cm

\noindent
{\bf Acknowledgements.} 
We acknowledge discussions with Juerong Feng, Benjamin Stamm, and Antoine Levitt. 
The project was initiated during a year-long research visit of DZ to the
University of British Columbia, supported by the China Scholarship Council (202406040146).
DZ was also supported by the National Natural Science Foundation of China (124B2020).
HC was supported by the National Natural Science Foundation of China (12371431). 
CO and BH were supported by the Natural Sciences and Engineering Research Council of Canada through NSERC Discovery grants and an NSERC-NSF Collaboration on quantum science and artificial intelligence grant. 
The research was also supported through computational resources and services provided by the Digital Research Alliance of Canada (alliancecan.ca) and by Advanced Research Computing at the University of British Columbia. 
CO is a partner in Symmetric Group LLP which licenses force fields commercially.

\appendix

\section{Relation with existing work}
\label{sec:related_work}

This section positions the loop-product invariant among several related
approaches to gauge freedom, Wannier gauges, and reduced representations
of parameter-dependent matrix families.

\subsection{Continuous trace-word invariants}
\label{sec:trace_word_relation}

We now relate the loop-product invariants to the classical trace-word perspective. 

Let $(A_1,\ldots,A_s)$ be a finite tuple of real-symmetric $\size\times\size$ matrices. 
The seminal work of Sibirskii and Procesi~\citep{sibirskii1968algebraic,procesi1976invariant} shows that this tuple is determined, up to simultaneous orthogonal conjugation, by trace words $ \operatorname{tr}(A_{i_1}A_{i_2}\cdots A_{i_k})$ for $i_j\in\{1,\ldots,s\}$. 
In characteristic zero, we can restrict to words of length at most $L_\size$, a universal bound for $\size\times\size$ matrices.
By the Procesi--Razmyslov bound one may take
$L_\size = \mathcal{O}(\size^2)$~\citep{razmyslov1974trace,procesi1976invariant}.
For a parameterized family $H:\DomX\to\SymHN$ and a subset $\Omega\subseteq\DomX$, we define the $k$-point continuous trace word as
\begin{equation*}
    T_k(\xx_1,\ldots,\xx_k;H) := \tr\bigl(H(\xx_1)H(\xx_2)\cdots H(\xx_k)\bigr), \qquad (\xx_1,\ldots,\xx_k)\in\Omega^k .
\end{equation*}
The corresponding trace-word map collects these words for all lengths up to $L_\size$:
\begin{equation*}
    \mathcal T_{\le L_\size,\Omega}(H) := \bigl(T_k(\cdot;H)\bigr)_{1\le k\le L_\size} \in \prod_{k=1}^{L_\size} C(\Omega^k,\R).
\end{equation*}
When $\Omega=\DomX$, we simply write $\mathcal T_{\le L_\size}(H) := \mathcal T_{\le L_\size,\DomX}(H)$. This trace-word map is clearly gauge-invariant; if $\QO \in \ON$ then  cyclicity of the trace ensures
\begin{equation*}
    T_k(\xx_1,\ldots,\xx_k; \QO^\trans H\QO)
    = 
    T_k(\xx_1,\ldots,\xx_k; H ). 
\end{equation*}
%
%
The converse—that it is in fact \emph{complete}—is the content of the next
result. 

\begin{proposition}[Completeness of continuous trace words]
\label{prop:continuous_trace_words_complete}
Let $H_1,H_2:\DomX\to\SymHN$. If $\mathcal T_{\le L_\size}(H_1) = \mathcal T_{\le L_\size}(H_2)$, then there exists a constant $\QO\in\ON$ such that
\begin{equation*}
    H_2(\xx)=\QO^\trans H_1(\xx)\QO \qquad \forall\,\xx\in\DomX .
\end{equation*}
\end{proposition}

\begin{proof}
Let $F=\{\xx^{(1)},\ldots,\xx^{(s)}\}\subset\DomX$ be an arbitrary finite set. 
Equality of continuous trace words implies that for every index sequence $i_1,\ldots,i_k\in\{1,\ldots,s\}$ ($1\le k\le L_\size$), we have
\begin{equation*}
    \tr\bigl(H_1(\xx^{(i_1)})\cdots H_1(\xx^{(i_k)})\bigr) = \tr\bigl(H_2(\xx^{(i_1)})\cdots H_2(\xx^{(i_k)})\bigr).
\end{equation*}
By the finite-dimensional trace-word theorem, the matrix tuples $\bigl(H_1(\xx^{(1)}),\ldots,H_1(\xx^{(s)})\bigr)$ and $\bigl(H_2(\xx^{(1)}),\ldots,H_2(\xx^{(s)})\bigr)$ are simultaneously conjugate. 
Thus, there exists $Q_F\in\ON$ such that $H_2(\xx)=Q_F^\trans H_1(\xx)Q_F$ for all $\xx\in F$.

For each finite $F\subset\DomX$, define the closed subset
\begin{equation*}
    \mathcal Q_F := \left\{ Q\in\ON : H_2(\xx)=Q^\trans H_1(\xx)Q \quad \forall\,\xx\in F \right\}.
\end{equation*}
We just showed that $\mathcal Q_F$ is nonempty. 
Furthermore, for any finite subsets $F_1,\ldots,F_m \subset \DomX$, their intersection $\bigcap_{j=1}^m \mathcal Q_{F_j} = \mathcal Q_{\bigcup F_j}$ is also nonempty. 
Therefore, the family $\{\mathcal Q_F\}$ has the finite intersection property. 
Since $\ON$ is compact, the global intersection over all finite $F\subset\DomX$ is nonempty, $\bigcap_{\substack{F\subset\DomX \\ F \text{ finite}}} \mathcal Q_F \neq\emptyset .$ 
Choosing $\QO$ in this intersection yields $H_2(\xx)=\QO^\trans H_1(\xx)\QO$ for all $\xx\in\DomX$.
\end{proof}

We now compare the trace-word invariant with the loop-product invariant on the
admissible class—families $H\in C^1(\DomX,\SymHN)$ that are non-degenerate with a
connected interaction graph at every $\xx\in\DomX$. 
On this class the reduced loop-product invariant is defined and complete
(\Cref{thm:complete_invariant_new}), and with the deterministic choices of
spanning tree and active directions (\Cref{rem:deterministic_map}) we denote it
simply by $\map(H)$. Under the conditions of \Cref{thm:complete_invariant_new}, both the trace words and the loop products are complete invariants, hence $\map(H_1)=\map(H_2)$ if and only if $\mathcal T_{\le L_\size}(H_1) = \mathcal T_{\le L_\size}(H_2)$. In that sense the two invariants are equivalent. Moreover, we show next how the pointwise loop product invariants can be constructed from tracewords. 




\begin{remark}[Pointwise differential content of trace words]
\label{rem:loop_products_from_trace_words}

For $H\in C^1(\DomX,\SymHN)$ with simple spectrum, the spectral projector onto the
eigenline of $\lambda_n(\xx)$ is
\begin{equation*}
    P_n(\xx) = \prod_{\ell\neq n} \frac{H(\xx)-\lambda_\ell(\xx)I}{\lambda_n(\xx)-\lambda_\ell(\xx)}
    = u_n(\xx)u_n(\xx)^\trans .
\end{equation*}
The eigenvalues are recovered from the diagonal trace words
$\tr\bigl(H(\xx)^r\bigr)=\sum_{n=1}^{\size}\lambda_n(\xx)^r$ via Newton's
identities, so each $P_n(\xx)$, being a polynomial in $H(\xx)$ with
trace-word coefficients, is itself determined by trace-word data.

Fix a $k$-loop $\bgamma=\bigl((e_1,\mu_1),\ldots,(e_k,\mu_k)\bigr)$ with
$e_r=\{m_r,m_{r+1}\}$ and $m_{k+1}=m_1$, and write
$J^\mu_{mn}(\xx;H):=\langle u_m(\xx),\partial_{x_\mu}H(\xx) u_n(\xx)\rangle$. 
Inserting $P_{m_r}=u_{m_r}u_{m_r}^\trans$ and using
$u_{m_r}^\trans u_{m_r}=1$ together with the cyclicity of the trace
($m_{k+1}=m_1$) gives
\begin{equation*}
    \tr\Big( P_{m_1}(\xx)\,\partial_{x_{\mu_1}}H(\xx) \cdots
              P_{m_k}(\xx)\,\partial_{x_{\mu_k}}H(\xx) \Big)
    = J^{\mu_1}_{m_1m_2}(\xx;H)\,J^{\mu_2}_{m_2m_3}(\xx;H)\cdots J^{\mu_k}_{m_km_1}(\xx;H).
\end{equation*}
Since $J^\mu_{mn}(\xx;H)=\bigl(\lambda_n(\xx)-\lambda_m(\xx)\bigr)\calA^\mu_{mn}(\xx;H)$
for $m\neq n$, dividing by the spectral gaps recovers the derivative-coupling loop
product,
\begin{equation*}
\begin{aligned}
    \Pi\bigl(\bgamma;\calA(\xx;H)\bigr)
    &= \Bigl[ \prod_{r=1}^k \bigl(\lambda_{m_{r+1}}(\xx)-\lambda_{m_r}(\xx)\bigr)^{-1} \Bigr] \times
      \tr\Big( \prod_{r=1}^k P_{m_r}(\xx)\,\partial_{x_{\mu_r}}H(\xx) \Big).
\end{aligned}
\end{equation*}
Each $P_{m_r}$ is a polynomial in $H(\xx)$, so the trace on the right is a fixed
parameter-derivative of the multi-point trace word
$T_k(\xx_1,\ldots,\xx_k;H)$ evaluated at $\xx_1=\cdots=\xx_k=\xx$. Thus multi-point
trace words contain every pointwise loop product.
\end{remark}

\Cref{rem:trace_words_from_loop_products_direct} establishes the reverse
direction: on the simple-spectrum, connected-graph class, the loop-product data
determine every continuous trace word. The construction integrates the
derivative-coupling connection into overlap matrices and inserts them into the
trace formula; the only freedom left by the loop products is an endpoint sign
gauge, which cancels inside every closed trace word. 

\begin{remark}[Computing trace words from loop products]
\label{rem:trace_words_from_loop_products_direct}

Let $H(\xx)=\eframe(\xx)\Lambda(\xx)\eframe(\xx)^\trans$ 
with $\Lambda(\xx)=\diag(\lambda_1(\xx),\ldots,\lambda_\size(\xx))$ and a $C^1$ ordered eigenframe
$\eframe=(u_1,\ldots,u_\size)$. The derivative couplings satisfy $\partial_{x_\mu}\eframe(\xx)=\eframe(\xx)\calA^\mu(\xx)$ and $\calA^\mu(\xx)=\eframe(\xx)^\trans\partial_{x_\mu}\eframe(\xx)$.  
Hence, if $\gamma:[0,1]\to\DomX$ is a piecewise $C^1$ path from
$\yy$ to $\xx$, the overlap matrix $W_\gamma(t):=\eframe(\yy)^\trans \eframe(\gamma(t))$ 
solves the linear equation 
\begin{equation*}
    \dot W_\gamma(t)
    =
    W_\gamma(t)
    \sum_{\mu=1}^{\dof}\dot\gamma_\mu(t)\calA^\mu(\gamma(t)),
    \qquad
    W_\gamma(0)=I .
\end{equation*}
Consequently,
\begin{equation*}
    W_\gamma(1)=\eframe(\yy)^\trans \eframe(\xx)
    =
    \mathcal P\exp\left(
        \int_\gamma \sum_{\mu=1}^{\dof}\calA^\mu\,dx_\mu
    \right),
\end{equation*}
where $\mathcal P$ denotes path ordering.

The pointwise loop products determine the sign-gauge orbit of the collection
$\{\calA^\mu(\xx)\}_{\mu=1}^{\dof}$. Thus the reconstructed connection is unique up to
\begin{equation*}
    \calA^\mu(\xx)\longmapsto S(\xx)\calA^\mu(\xx)S(\xx),
    \qquad
    S(\xx)=\diag(\sigma_1(\xx),\ldots,\sigma_\size(\xx)).
\end{equation*}
This changes the corresponding overlap matrix only by
endpoint signs, i.e., $\eframe(\yy)^\trans \eframe(\xx)
    \longmapsto
    S(\yy)\,\eframe(\yy)^\trans \eframe(\xx)\,S(\xx)$. 

Now fix $k$ parameter values
$\xx_1,\ldots,\xx_k\in\DomX$, and write $W_{r,r+1}:=\eframe(\xx_r)^\trans \eframe(\xx_{r+1})$ with $\xx_{k+1}:=\xx_1$.  
Inserting the spectral decomposition of each $H(\xx_r)$ gives the exact
formula
\begin{equation*}
    \begin{aligned}
    T_k(\xx_1,\ldots,\xx_k;H)
    &=
    \tr\bigl(H(\xx_1)\cdots H(\xx_k)\bigr)
    \\
    &=
    \tr\Bigl(
        \Lambda(\xx_1)W_{1,2}
        \Lambda(\xx_2)W_{2,3}
        \cdots
        \Lambda(\xx_k)W_{k,1}
    \Bigr)\\
    &= \sum_{n_1,\ldots,n_k}
    \prod_{r=1}^{k}
    \lambda_{n_r}(\xx_r)
    \bigl(W_{r,r+1}\bigr)_{n_r n_{r+1}},
    \qquad n_{k+1}:=n_1 .
\end{aligned}
\end{equation*}
This expression is independent of the sign choices in the eigenframes. 
Therefore the spectral data together with the loop-product data determine every
continuous trace word $T_k(\xx_1,\ldots,\xx_k;H)$.
\end{remark}

In summary, continuous trace words provide a basis-free complete invariant without requiring simple spectra or connected interaction graphs. 
However, they are multi-point and non-local objects: evaluating a $k$-point trace word on $s$ sampled configurations requires $\mathcal{O}(s^k)$ operations. 
In contrast, the loop-product invariant is strictly pointwise. 
Through spanning-tree reduction, its evaluation cost grows only linearly with the number of configurations and explicitly adapts to the interaction graph's sparsity.

\subsection{Cycle products in graph-gauge problems}
\label{subsec:relation_graph_gauge}

Our loop-product construction relies on a standard principle from graph-gauge theory: variables on individual edges are usually gauge-dependent, while their products along closed cycles are gauge-invariant. 
A well-known example is the theory of magnetic Schr\"odinger operators on discrete periodic graphs \citep{korotyaev2017magnetic}. 
There, a magnetic potential assigns a phase to each edge. 
Changing the phase at the vertices acts as a gauge transformation, modifying the edge phases without altering the physical magnetic operator. 
The true gauge-invariant quantities are the magnetic fluxes through closed cycles. 
By fixing a spanning tree, one can use a vertex gauge to set all tree-edge phases to zero; the fundamental cycles then capture all remaining physical information. 
These cycle fluxes are used to study how the magnetic field affects the spectrum, such as determining bandwidths.

Our loop products use this exact same cancellation mechanism, but for a different mathematical object. 
In our problem, the ``edge variables" are not magnetic phases, but the derivative-coupling entries $\calA^\mu_{mn}(\xx;H)$ between instantaneous eigenvectors. 
Changing the signs of these eigenvectors acts as a discrete vertex gauge:
\begin{equation*}
    \calA^\mu_{mn}(\xx;H) \longmapsto \sigma_m(\xx)\sigma_n(\xx)\calA^\mu_{mn}(\xx;H), \qquad \sigma_m(\xx)\in\{\pm1\}.
\end{equation*}
A single coupling entry changes sign under this gauge. 
However, the product of coupling entries around any closed loop is invariant, because every vertex sign appears exactly twice and squares to $1$. 
The spanning-tree reduction in \Cref{subsec:complete_invariant} is this graph-gauge principle adapted for our inverse problem. 
By choosing a spanning tree, the signs on the tree edges can always be absorbed by vertex signs. 
The only genuine sign obstructions are the holonomies around non-tree edges.
This is precisely what the fundamental cycle products measure. 

These are invariants of the eigenframe gauge --- the vertex signs above or their $\mathsf U(1)$ analogue in the self-adjoint case --- rather than of the ambient change of basis of \Cref{def:gauge_equivalence}; in this sense they are the classical predecessors of the loop products used in the present work.


\subsection{Relation with Wannier gauges}
\label{subsec:relation_wannier}

In electronic
structure, Bloch eigenstates are not uniquely determined by the Hamiltonian.
For a single isolated band, each Bloch state can be multiplied by a
$k$-dependent phase. 
For a composite group of bands, the freedom becomes a
$k$-dependent unitary rotation within the chosen band subspace.
The
construction of Wannier functions uses this gauge freedom to obtain localized
real-space orbitals, for instance through the maximally localized Wannier
function procedure \citep{marzari2012maximally}.

The common feature is the presence of an eigenframe gauge. 
Both settings
recognize that eigenvectors themselves are not physical observables. 
The
difference is the role played by the gauge. Wannier methods attempt to
choose a good gauge, usually a smooth and localized one, because the
resulting Wannier representation is meant to be used as a computational and
physical basis. 
By contrast, the present work does not attempt to select a
preferred eigenframe. 
Instead, it constructs quantities that are independent
of the eigenframe from the outset.

\subsection{Relation with reduced-basis models for parametric eigenproblems}
\label{subsec:relation_reduced_basis}
A further related direction is model order reduction for parametric eigenvalue problems. 
Recent work develops
certified projection-based reduced models for the smallest eigenvalue and its
eigenspace of large-scale parametric symmetric matrices
\citep{manucci2025certified}. 
%
Their construction uses weak greedy strategies,
spectral-gap estimates, and computable conditions ensuring that the
multiplicity of the smallest eigenvalue is captured in the reduced space.
A subsequent Taylor-reduced basis method \citep{stamm2026model} constructs a local
reduced space from derivatives of the spectral projector at a reference
parameter and relates this construction to multivariate analytic perturbation
theory. 
Related work \citep{garrigue2024reduced} studies
subspace-projection approximations of self-adjoint eigenvalue problems,
including degenerate cases and reduced spaces enriched by parameter
derivatives of eigenfunctions. 
\citet{manucci2024uniform} consider the uniform approximation of the smallest eigenvalue
of a large parameter-dependent symmetric matrix by a projected smaller
problem, using greedy subspace construction and proving uniform convergence
of the projected eigenvalue function. 

These works and the present paper concern parameter-dependent symmetric matrices, but the mathematical questions are different. 
Reduced-basis methods start from a large family $A(\xx)\in\C^{\size_{\rm full}\times \size_{\rm full}}$
and construct a low-dimensional subspace $V\in\C^{\size_{\rm full}\times r}$
so that the projected matrix $A_r(\xx):=V^\dagger A(\xx)V$ 
accurately approximates selected eigenvalues or eigenspaces. The main issues
are approximation error, certification, spectral gaps, degeneracies, and the
efficient assembly of the reduced space.

Our focus is instead an identifiability problem for a given finite-dimensional
matrix family. 
We ask when two families
represent the same physical object up to a constant orthogonal change of
basis. 
Nevertheless, the gauge viewpoint is relevant to reduced models as
well. If the reduced basis $V$ is replaced by another basis of the same
reduced subspace, $V\mapsto VU$ with $U\in\Ug{r}$, then $ A_r(\xx)
    \longmapsto
    U^\dagger A_r(\xx)U $. 
Thus a reduced Hamiltonian also has a representation-basis ambiguity. 
The
present invariant framework can therefore be viewed as a complementary tool:
it does not construct a reduced space, but it provides gauge-invariant
features and quotient-level distances for comparing finite-dimensional
Hamiltonian families or reduced Hamiltonian models.

In particular, reduced-basis methods aim to make large spectral computations
cheaper, whereas the loop-product invariant aims to make comparison and
identification of matrix families basis-independent. These two viewpoints
are compatible. A reduced model may be used to obtain an efficient
finite-dimensional representative, while the loop-product framework can be
used to compare such representatives without committing to a particular
basis inside the reduced space.

\section{Proof of \Cref{thm:full_stability}}
\label{app:proof_reduction}
We retain the setting and notation of \Cref{subsec:well_posedness}: the reference family $H_1$ is non-degenerate on the compact, rectifiably path-connected set $K$ with path-length bound $L_K$; the uniformity constants $g_K$, $B_K$, and $\kappa_K$ are those of \eqref{eq:stab:gap}, \eqref{eq:stab:upperbnd}, and \eqref{eq:kappa_dir_lower_bound}; $\Delta_K$ and $\dist_K$ are the invariant discrepancy and the equivalence distance defined there; and $\map_{\tree,\ell}$ is the associated complete invariant. We also use the pointwise derivative coupling $\calA(\xx;H)$ (\Cref{def:ptwise_derivative_coupling}), the loop products $\Pi(\bgamma;\cdot)$ along cycles $\bgamma$ (\Cref{def:mixed_loop_product}), and the sign-gauge group $\constS$ (\Cref{sec:berry}); two derivative couplings of the same family differ by conjugation with an element of $\constS$ (\Cref{lem:ptwise_sign_gauge_A}). Throughout, $C>0$ denotes a generic constant depending only on $g_K,B_K,\kappa_K,L_K,\size,\dof$; its value may change from line to line.

\Cref{lem:stab_H2_nondeg} shows that the competitor $H_2$ is also non-degenerate once the invariant discrepancy is sufficiently small; local ordered eigenframes are then available for both $H_1$ and $H_2$.

\begin{lemma}[Non-degeneracy of $H_2$]
\label{lem:stab_H2_nondeg}
Let $K\subset\DomX$ be compact, and suppose that $H_1$ has a uniform spectral gap on $K$, 
\begin{equation*}
    \min_{\xx\in K}\min_{j\neq k}
    |\lambda_j(\xx;H_1)-\lambda_k(\xx;H_1)|
    \ge g_K>0 .
\end{equation*}
If $\Delta_K(H_1,H_2)\le \frac{g_K}{4}$, then $H_2$ is non-degenerate on $K$, and
\begin{equation*}
    \min_{\xx\in K}\min_{j\neq k}
    |\lambda_j(\xx;H_2)-\lambda_k(\xx;H_2)|
    \ge \frac{g_K}{2}.
\end{equation*}
\end{lemma}
\begin{proof}
    This result follows from a straightforward computation.
\end{proof}

Near any $\xx_0 \in K$, both $H_1$ and $H_2$ therefore admit $C^1$ ordered eigenframes $\eframe_i(\xx,\xx_0)$ with continuous local derivative couplings
\begin{equation*}
    \tilde{\calA}^{\mu}(\xx,\xx_0;H_i)
    := \eframe_i(\xx,\xx_0)^\trans \partial_{x_\mu} \eframe_i(\xx,\xx_0) .
\end{equation*}
However, since $(\tree(\xx),\ell(\xx))$ varies with $\xx$, the mapping $\xx \mapsto \map_{\tree(\xx),\ell(\xx)}(\xx;H_i)$ need not be continuous.
To resolve this, we fix the tree and active labels at $\xx_0$ and consider the frozen invariant $\map_{\tree(\xx_0), \ell(\xx_0)}(\xx;H_i)$ for $\xx$ in a small neighbourhood of $\xx_0$; with the tree and labels fixed, this map is continuous in $\xx$.
At the centre $\xx = \xx_0$, the frozen choice coincides with the pointwise choice used in the definition of
$\Delta_K(H_1,H_2)$. Hence the frozen discrepancy 
\begin{equation*}
\Delta\map(\xx,\xx_0)
:=
\map_{\tree(\xx_0),\ell(\xx_0)}(\xx;H_2)
-
\map_{\tree(\xx_0),\ell(\xx_0)}(\xx;H_1)
\end{equation*}
is controlled at $\xx_0$ by $\Delta_K(H_1,H_2)$, and continuity propagates this control to a neighbourhood of $\xx_0$. Choosing this neighbourhood sufficiently small ensures that the frozen active tree entries remain uniformly bounded away from zero, as in \eqref{eq:kappa_dir_lower_bound}.

\begin{lemma}[Frozen invariant discrepancy]
\label{lem:stab_freezing}
Fix $\xx_0\in K$. Let $(\tree(\xx_0), \ell(\xx_0))$ be the pointwise tree and active labels at $\xx_0$. Let $W_{\xx_0} \subset \DomX$ be an open neighbourhood of $\xx_0$ on which
the local ordered eigenframes $\eframe_i(\xx,\xx_0)$ and their derivative couplings $\tilde{\calA}^{\mu}(\xx,\xx_0;H_i)$ are defined and continuous.
Assume that $\|\Delta\map(\xx_0,\xx_0)\|_\infty\le \delta$ and that, for every $e\in \edge(\tree(\xx_0))$, 
\begin{equation*}
    \left|\tilde{\calA}_e^{\ell_e(\xx_0)}(\xx_0,\xx_0;H_1)\right|\ge \kappa_K
    \quad \text{and} \quad 
    \left|
    \left(\tilde{\calA}_e^{\ell_e(\xx_0)}(\xx_0,\xx_0;H_2)\right)^2
    -
    \left(\tilde{\calA}_e^{\ell_e(\xx_0)}(\xx_0,\xx_0;H_1)\right)^2
    \right|
    \le \delta .
\end{equation*}
If $0<\delta\le 7\kappa_K^2/16$, then there exists an open ball
$U_{\xx_0} \subset W_{\xx_0}$ centred at $\xx_0$ such that $\|\Delta\map(\cdot,\xx_0)\|_{C^0(U_{\xx_0})}
\le 2\delta$, 
and, for all $\xx\in U_{\xx_0}$ and $e\in \edge(\tree(\xx_0))$, 
\begin{equation*}
    \left|\tilde{\calA}_e^{\ell_e(\xx_0)}(\xx,\xx_0;H_1)\right|
\ge \frac{\kappa_K}{2}
\quad \text{and} \quad 
\left|\tilde{\calA}_e^{\ell_e(\xx_0)}(\xx,\xx_0;H_2)\right|
\ge \frac{\kappa_K}{4}.
\end{equation*}
\end{lemma}

\begin{proof}
By the assumed continuity, $\Delta\map(\cdot,\xx_0)$ is continuous on $W_{\xx_0}$, so there is
an open ball $U_1\subset W_{\xx_0}$ centred at $\xx_0$ on which it differs from its value at
$\xx_0$ by at most $\delta$ in $\|\cdot\|_\infty$; together with the assumption
$\|\Delta\map(\xx_0,\xx_0)\|_\infty\le\delta$ this gives $\|\Delta\map(\cdot,\xx_0)\|_{C^0(U_1)}\le 2\delta$.
For the active entries, fix $e\in\edge(\tree(\xx_0))$ and abbreviate
$a_i:=\tilde{\calA}_e^{\ell_e(\xx_0)}(\xx_0,\xx_0;H_i)$. The hypotheses give $|a_1|\ge\kappa_K$ 
and $|a_2^2-a_1^2|\le\delta$, hence $|a_2|^2\ge\kappa_K^2-\delta\ge\tfrac{9}{16}\kappa_K^2$ and
therefore $|a_2|\ge\tfrac{3\kappa_K}{4}$.
Since $\edge(\tree(\xx_0))$ is finite and each $\xx\mapsto\tilde{\calA}_e^{\ell_e(\xx_0)}(\xx,\xx_0;H_i)$
is continuous on $W_{\xx_0}$, we may shrink $U_1$ to an open ball $U_{\xx_0}$ centred at $\xx_0$ on
which $|\tilde{\calA}_e^{\ell_e(\xx_0)}(\xx,\xx_0;H_1)|\ge\kappa_K/2$ and
$|\tilde{\calA}_e^{\ell_e(\xx_0)}(\xx,\xx_0;H_2)|\ge\kappa_K/4$ for every $e\in\edge(\tree(\xx_0))$.
Shrinking preserves the discrepancy estimate.
\end{proof}

\Cref{lem:stab_freezing} bounds the frozen invariant discrepancy by $2\delta$ on a ball $U_{\xx_0}$ around $\xx_0$, and keeps the frozen active tree entries bounded away from zero there. 
\Cref{lem:stab_local_reconstruction} shows that
the two local derivative-coupling families agree up to a single diagonal sign gauge, with an error controlled by the frozen invariant discrepancy; the quantitative analogue of {\it Steps 1-4} in the proof of \Cref{thm:complete_invariant_new}.
The families $A_i^\mu$ are kept abstract: in the assembly below they are instantiated as the local derivative couplings $\tilde{\calA}^\mu(\cdot,\xx_0;H_i)$ furnished by \Cref{lem:stab_freezing}, but the proof uses only their skew-symmetry and the invariant-discrepancy bounds~(i)--(v). The lemma can be read as a purely quantitative statement about two families of skew-symmetric matrices.

%
%
\begin{lemma}[Local sign reconstruction from invariant discrepancies]
\label{lem:stab_local_reconstruction}
Let $U\subset\R^{\dof}$ be connected and let $V\subset U$. Let $\tree$ be a spanning tree on $\{1,\ldots,\size\}$, and assign an active label $\ell_e\in\{1,\ldots,\dof\}$ to each tree edge $e\in \edge(\tree)$. Fix an orientation of every tree edge. 
Let $A_i^\mu(\xx)$, $i=1,2$ and $\mu=1,\ldots,\dof$, be continuous families of real skew-symmetric matrices. Suppose that, for some $0\le\varepsilon\le1$, $\kappa>0$, and $B_K>0$, the following hold. 

\begin{enumerate}[label=(\roman*), leftmargin=*, widest=iii]
\item \emph{Boundedness on $V$:} 
\begin{equation*}
    |A_{1,mn}^\mu(\xx)|\le B_K,
    \qquad
    \xx\in V,\quad \mu=1,\ldots,\dof,\quad m,n=1,\ldots,\size .
\end{equation*}

\item \emph{Tree activation on $U$:}
\begin{equation*}
    |A_{i,e}^{\ell_e}(\xx)|\ge \kappa,
    \qquad
    \xx\in U,\quad i=1,2,\quad e\in \edge(\tree).
\end{equation*}

\item \emph{Pure two-loop on $V$:}
\begin{equation*}
    |A_{2,mn}^\mu(\xx)^2-A_{1,mn}^\mu(\xx)^2|
    \le \varepsilon, \qquad \xx\in V, \quad \mu=1,\ldots,\dof, \quad m,n=1,\ldots,\size. 
\end{equation*}

\item \emph{Mixed two-loop on tree edges on $V$:}
\begin{equation*}
    |A_{2,e}^{\ell_e}(\xx)A_{2,e}^\mu(\xx)
      -A_{1,e}^{\ell_e}(\xx)A_{1,e}^\mu(\xx)|
    \le \varepsilon, \qquad \xx\in V, \quad e\in \edge(\tree), \quad \mu=1,\ldots,\dof. 
\end{equation*}

\item \emph{Fundamental cycles on $V$:}
For every non-tree pair $\{p,q\}\notin \edge(\tree)$, let $p=v_0,\ldots,v_k=q$ be
the unique tree path from $p$ to $q$ and set $f_r:=\{v_{r-1},v_r\}$; then, for all
$\xx\in V$ and $\mu=1,\ldots,\dof$,
\begin{equation*}
\begin{aligned}
    \Bigl|\;
        &A_{2,qp}^\mu(\xx)\prod_{r=1}^k A_{2,v_{r-1}v_r}^{\ell_{f_r}}(\xx) -
         A_{1,qp}^\mu(\xx)\prod_{r=1}^k A_{1,v_{r-1}v_r}^{\ell_{f_r}}(\xx)
    \;\Bigr| \le \varepsilon .
\end{aligned}
\end{equation*}
\end{enumerate}
Then there exists a diagonal sign matrix $S\in\constS$ such that $\max_{1\le\mu\le\dof}\|A_2^\mu(\cdot)-SA_1^\mu(\cdot) S\|_{C^0(V)}\le C\,\varepsilon$, 
where $C$ depends only on $\kappa$, $B_K$, $\size$, and $\dof$. 
\end{lemma}

\begin{proof}
We first record a uniform bound for the entries of $A_2^\mu(\cdot)$ on $V$. By the pure two-loop assumption~(iii) and $0\le \varepsilon\le 1$,
\begin{equation*}
    |A_{2,mn}^\mu(\xx)|^2
    \le
    |A_{1,mn}^\mu(\xx)|^2+\varepsilon
    \le
    B_K^2+1,
    \qquad
    \xx\in V.
\end{equation*}
Thus all entries of $A_1^\mu(\cdot)$ and $A_2^\mu(\cdot)$ are bounded on $V$ by a constant $B_\sharp$ depending only on $B_K$.

We now reconstruct the signs on the tree. Fix a tree edge $e\in\edge(\tree)$. By the tree activation assumption~(ii),
\begin{equation*}
    |A_{i,e}^{\ell_e}(\xx)|\ge \kappa,
    \qquad
    \xx\in U,\quad i=1,2.
\end{equation*}
Hence the function
\begin{equation*}
    s_e(\xx):=\operatorname{sgn}\!\left(\frac{A_{2,e}^{\ell_e}(\xx)}{A_{1,e}^{\ell_e}(\xx)}\right)
\end{equation*}
is well-defined and continuous on $U$, with values in $\{\pm1\}$. Since $U$ is connected, $s_e(\xx)$ is constant on $U$; we denote this constant by $s_e$.

For $\xx\in V$, the two quantities $A_{2,e}^{\ell_e}(\xx)$ and $s_eA_{1,e}^{\ell_e}(\xx)$ have the
same sign. Therefore 
\begin{equation*}
    |A_{2,e}^{\ell_e}(\xx)+s_eA_{1,e}^{\ell_e}(\xx)|
    =
    |A_{2,e}^{\ell_e}(\xx)|+|A_{1,e}^{\ell_e}(\xx)|
    \ge 2\kappa.
\end{equation*}
Using the pure two-loop discrepancy with $\mu=\ell_e$, we obtain
\begin{equation*}
    |A_{2,e}^{\ell_e}(\xx)-s_eA_{1,e}^{\ell_e}(\xx)| = \frac{|A_{2,e}^{\ell_e}(\xx)^2-A_{1,e}^{\ell_e}(\xx)^2|}
         {|A_{2,e}^{\ell_e}(\xx)+s_eA_{1,e}^{\ell_e}(\xx)|}\le \frac{\varepsilon}{2\kappa}.
\end{equation*}
Hence
\begin{equation}
\label{eq:local_recon_tree_active}
    \|A_{2,e}^{\ell_e}(\cdot)-s_eA_{1,e}^{\ell_e}(\cdot)\|_{C^0(V)}
    \le C\varepsilon .
\end{equation}

Next, fix a direction $\mu\in\{1,\ldots,\dof\}$. The mixed two-loop assumption~(iv) gives
\begin{equation*}
    |A_{2,e}^{\ell_e}(\xx)A_{2,e}^\mu(\xx)
      -A_{1,e}^{\ell_e}(\xx)A_{1,e}^\mu(\xx)|
    \le \varepsilon,
    \qquad \xx\in V.
\end{equation*}
Using the identity
\begin{equation*}
    \begin{aligned}
    &A_{2,e}^{\ell_e}(\xx)\left(A_{2,e}^\mu(\xx)-s_eA_{1,e}^\mu(\xx)\right) \\
    &\quad = \left(A_{2,e}^{\ell_e}(\xx)A_{2,e}^\mu(\xx)-A_{1,e}^{\ell_e}(\xx)A_{1,e}^\mu(\xx)\right) 
     + \left(A_{1,e}^{\ell_e}(\xx)-s_eA_{2,e}^{\ell_e}(\xx)\right)A_{1,e}^\mu(\xx),
\end{aligned}
\end{equation*}   
together with $|A_{2,e}^{\ell_e}(\xx)|\ge\kappa$, $|A_{1,e}^\mu(\xx)|\le B_K$, and 
\eqref{eq:local_recon_tree_active}, we obtain
\begin{equation*}
    \|A_{2,e}^\mu(\cdot)-s_eA_{1,e}^\mu(\cdot)\|_{C^0(V)}
    \le C\varepsilon
\end{equation*}
for every $e\in\edge(\tree)$ and every $\mu=1,\ldots,\dof$. Since the matrices are skew-symmetric, the same estimate holds if the orientation of the edge is reversed.

Choose a root $v_\ast$ of the tree. For each vertex $v$, define $\sigma_v:=\prod_{f\in P_{\tree}(v_\ast,v)} s_f$, where $P_{\tree}(v_\ast,v)$ is the unique tree path from $v_\ast$ to $v$. 
Set $S:=\diag(\sigma_1,\ldots,\sigma_\size)\in\constS$.
Then, for every tree edge $\{p,q\}\in\edge(\tree)$, $s_{\{p,q\}}=\sigma_p\sigma_q$. 
Consequently,
\begin{equation}
\label{eq:local_recon_tree_edges}
    |A_{2,pq}^\mu(\xx)-\sigma_p\sigma_qA_{1,pq}^\mu(\xx)|
    \le C\varepsilon,
    \qquad
    \xx\in V,
\end{equation}
for every tree edge $\{p,q\}\in\edge(\tree)$ and every
$\mu=1,\ldots,\dof$.

It remains to control the non-tree pairs. Let
$\{p,q\}\notin\edge(\tree)$, and let $p=v_0,\ldots,v_k=q$ be the unique tree path from $p$ to $q$. Set $f_r:=\{v_{r-1},v_r\}$ for $r=1,\ldots,k$. Define the active path products
\begin{equation*}
    \Pi_i(\xx)
    :=
    \prod_{r=1}^k
    A_{i,v_{r-1}v_r}^{\ell_{f_r}}(\xx),
    \qquad i=1,2.
\end{equation*}
By tree activation, $|\Pi_2(\xx)|
    \ge
    \kappa^k
    \ge
    \min\{\kappa,\kappa^{\size-1}\}
    =:c_{\kappa,\size}>0$. 
Moreover, applying the tree-edge estimate to each edge of the
path gives
\begin{equation*}
\left|
A_{2,v_{r-1}v_r}^{\ell_{f_r}}(\xx)
-
\sigma_{v_{r-1}}\sigma_{v_r}
A_{1,v_{r-1}v_r}^{\ell_{f_r}}(\xx)
\right|
\le C\varepsilon , \qquad r=1,\ldots,k .
\end{equation*}
Since all factors are bounded on $V$ by $B_\sharp$, the elementary product estimate
\begin{equation*}
    \left|
    \prod_{r=1}^k y_r-\prod_{r=1}^k z_r
    \right|
    \le
    \sum_{r=1}^k
    |y_r-z_r|
    \prod_{s<r}|y_s|
    \prod_{s>r}|z_s|
\end{equation*}
implies 
\begin{equation*}
    |\Pi_2(\xx)-\sigma_p\sigma_q\Pi_1(\xx)|
    \le C\varepsilon,
    \qquad \xx\in V.
\end{equation*}

Now fix $\mu\in\{1,\ldots,\dof\}$. The fundamental-cycle assumption~(v) gives
\begin{equation*}
    |A_{2,qp}^\mu(\xx)\Pi_2(\xx)-A_{1,qp}^\mu(\xx)\Pi_1(\xx)| \le \varepsilon, \quad \xx\in V.
\end{equation*}
Using the identity 
\begin{align*}
    &\Pi_2(\xx) \bigl(A_{2,qp}^\mu(\xx)-\sigma_p\sigma_qA_{1,qp}^\mu(\xx)\bigr) \\
    =& \bigl(A_{2,qp}^\mu(\xx)\Pi_2(\xx)-A_{1,qp}^\mu(\xx)\Pi_1(\xx)\bigr) + A_{1,qp}^\mu(\xx) \bigl(\Pi_1(\xx)-\sigma_p\sigma_q\Pi_2(\xx)\bigr), 
\end{align*}
together with $|A_{1,qp}^\mu(\xx)|\le B_K$, the product estimate above, and the lower bound $|\Pi_2(\xx)|\ge c_{\kappa,\size}$, we obtain $|A_{2,qp}^\mu(\xx)-\sigma_p\sigma_qA_{1,qp}^\mu(\xx)| \le C\varepsilon$ for $\xx\in V$. Since $A_{i,qp}^\mu(\xx) = -A_{i,pq}^\mu(\xx)$, this is equivalent to
\begin{equation}
\label{eq:local_recon_nontree_edges}
    |A_{2,pq}^\mu(\xx)-\sigma_p\sigma_qA_{1,pq}^\mu(\xx)|
    \le C\varepsilon,
    \qquad \xx\in V,
\end{equation}
for every non-tree pair $\{p,q\}\notin\edge(\tree)$ and every $\mu=1,\ldots,\dof$.

Combining \eqref{eq:local_recon_tree_edges} and
\eqref{eq:local_recon_nontree_edges}, we obtain
\begin{equation*}
    |A_{2,pq}^\mu(\xx)-\sigma_p\sigma_qA_{1,pq}^\mu(\xx)|
    \le C\varepsilon
\end{equation*}
for every $p\neq q$, every $\mu=1,\ldots,\dof$, and every $\xx\in V$.
The diagonal entries vanish because the matrices are skew-symmetric. Since $(SA_1^\mu(\xx)S)_{pq} = \sigma_p\sigma_qA_{1,pq}^\mu(\xx)$, we conclude that
\begin{equation*}
    \max_{1\le\mu\le\dof}
    \|A_2^\mu(\cdot)-SA_1^\mu(\cdot)S\|_{C^0(V)}
    \le C\varepsilon,
\end{equation*}
with $C$ depending only on $\kappa$, $B_K$, $\size$, and $\dof$.
\end{proof}

The sign matrices of \Cref{lem:stab_local_reconstruction} are constructed independently on each neighbourhood and need not fit together into a single global gauge. \Cref{lem:stab_overlap_compatibility} shows that, on every nonempty overlap, the associated local orthogonal gauges agree up to one overall sign. This corresponds to the first half of {\it Step 6} in the proof of \Cref{thm:complete_invariant_new}.

\begin{lemma}[Overlap compatibility of local gauges]
\label{lem:stab_overlap_compatibility}
Let $U_\alpha$ and $U_\beta$ be two open balls such that $V_{\alpha\beta}:=U_\alpha\cap U_\beta\cap K\neq\emptyset$, and assume that $U_\alpha\cap U_\beta$ is connected. For $i=1,2$, let
$\eframe_i^{(\alpha)}$ and $\eframe_i^{(\beta)}$ be ordered $C^1$ eigenframes of $H_i$ on $U_\alpha$ and $U_\beta$, respectively. Define $\tilde{\calA}_i^{(\alpha)\mu} := \eframe_i^{(\alpha)\trans}\partial_{x_\mu}\eframe_i^{(\alpha)}$ and $\tilde{\calA}_i^{(\beta)\mu} := \eframe_i^{(\beta)\trans}\partial_{x_\mu}\eframe_i^{(\beta)}$. 
Let $S_\alpha,S_\beta\in\constS$. Assume that, for some $\varepsilon>0$, 
\begin{equation*}
    \max_{1\le\mu\le\dof} \left\| \tilde{\calA}_2^{(\gamma)\mu} - S_\gamma\tilde{\calA}_1^{(\gamma)\mu}S_\gamma \right\|_{C^0(V_{\alpha\beta})} \le \varepsilon, \quad \gamma\in\{\alpha,\beta\}. 
\end{equation*}
Assume also that there is a spanning tree $\tree_\alpha$ with active labels $\ell_{\alpha,e}$ such that $|(\tilde{\calA}_1^{(\alpha)\ell_{\alpha,e}})_e(\xx)| \ge \kappa_0$ for $\xx\in V_{\alpha\beta}$ and $e\in \edge(\tree_\alpha)$. 
If $\varepsilon<\kappa_0$, then the local gauges $Q^{(\alpha)} := \eframe_1^{(\alpha)}S_\alpha\eframe_2^{(\alpha)\trans}$ and $Q^{(\beta)} := \eframe_1^{(\beta)}S_\beta\eframe_2^{(\beta)\trans}$ satisfy $Q^{(\beta)}=\pm Q^{(\alpha)}$ on $U_\alpha \cap U_\beta$. 
\end{lemma}

\begin{proof}
On $U_\alpha\cap U_\beta$, define $T_i^{(\alpha\beta)}(\xx) := \eframe_i^{(\alpha)}(\xx)^\trans \eframe_i^{(\beta)}(\xx)$ for $i = 1,2$. 
Since $\eframe_i^{(\alpha)}$ and $\eframe_i^{(\beta)}$ are ordered eigenframes of the same simple-spectrum matrix $H_i(\xx)$, they can differ only by the signs of their columns. Hence
$T_i^{(\alpha\beta)}(\xx)\in \constS$. Moreover, $T_i^{(\alpha\beta)}$ is continuous. Since
$U_\alpha\cap U_\beta$ is connected and $\constS$ is discrete, $T_i^{(\alpha\beta)}(\xx)$ is constant on $U_\alpha\cap U_\beta$; we write this constant as $T_i^{(\alpha\beta)}$. Thus $\eframe_i^{(\beta)} = \eframe_i^{(\alpha)}T_i^{(\alpha\beta)}$ on $U_\alpha\cap U_\beta$. 
Consequently, $\tilde{\calA}_i^{(\beta)\mu} = T_i^{(\alpha\beta)} \tilde{\calA}_i^{(\alpha)\mu} T_i^{(\alpha\beta)}$. 

The assumed estimate on $U_\beta$ gives, for each $\mu$,
\begin{equation*}
    \tilde{\calA}_2^{(\beta)\mu} = S_\beta \tilde{\calA}_1^{(\beta)\mu} S_\beta + E_\beta^\mu, \qquad \|E_\beta^\mu\|_{C^0(V_{\alpha\beta})}\le \varepsilon .
\end{equation*}
Using the transition identities and conjugating by $T_2^{(\alpha\beta)}$, we
obtain on $V_{\alpha\beta}$
\begin{equation*}
\begin{aligned}
    \tilde{\calA}_2^{(\alpha)\mu}
    &= T_2^{(\alpha\beta)} \tilde{\calA}_2^{(\beta)\mu} T_2^{(\alpha\beta)}  \\
    &= T_2^{(\alpha\beta)} S_\beta \tilde{\calA}_1^{(\beta)\mu} S_\beta T_2^{(\alpha\beta)} + T_2^{(\alpha\beta)}E_\beta^\mu T_2^{(\alpha\beta)}  \\
    &= T_2^{(\alpha\beta)} S_\beta T_1^{(\alpha\beta)} \tilde{\calA}_1^{(\alpha)\mu} T_1^{(\alpha\beta)} S_\beta T_2^{(\alpha\beta)} + \widehat E_\beta^\mu ,
\end{aligned}
\end{equation*}
where $\widehat E_\beta^\mu := T_2^{(\alpha\beta)}E_\beta^\mu T_2^{(\alpha\beta)}$ satisfies $\|\widehat E_\beta^\mu\|_{C^0(V_{\alpha\beta})}\le \varepsilon$. 
All sign matrices are diagonal, hence commute; define $\widehat S_\beta^{(\alpha)} := T_2^{(\alpha\beta)}S_\beta T_1^{(\alpha\beta)} \in\constS$. 
Then the $U_\beta$-estimate becomes $\tilde{\calA}_2^{(\alpha)\mu} = \widehat S_\beta^{(\alpha)} \tilde{\calA}_1^{(\alpha)\mu} \widehat S_\beta^{(\alpha)} + \widehat E_\beta^\mu$. On the other hand, the assumed estimate on $U_\alpha$ gives $\tilde{\calA}_2^{(\alpha)\mu} = S_\alpha \tilde{\calA}_1^{(\alpha)\mu} S_\alpha + E_\alpha^\mu$ with $\|E_\alpha^\mu\|_{C^0(V_{\alpha\beta})}\le \varepsilon$. 
Comparing the two representations of $\tilde{\calA}_2^{(\alpha)\mu}$, we get
\begin{equation*}
    \max_{1\le\mu\le\dof} \left\| S_\alpha\tilde{\calA}_1^{(\alpha)\mu}S_\alpha - \widehat S_\beta^{(\alpha)} \tilde{\calA}_1^{(\alpha)\mu} \widehat S_\beta^{(\alpha)} \right\|_{C^0(V_{\alpha\beta})} \le 2\varepsilon. 
\end{equation*}

Set $R:=S_\alpha\widehat S_\beta^{(\alpha)} = \diag(r_1,\ldots,r_\size)\in\constS$. 
Multiplying the last estimate on the left and right by $S_\alpha$ preserves the norm and gives
\begin{equation*}
    \max_{1\le\mu\le\dof} \left\| \tilde{\calA}_1^{(\alpha)\mu} - R\tilde{\calA}_1^{(\alpha)\mu}R \right\|_{C^0(V_{\alpha\beta})} \le 2\varepsilon .
\end{equation*}
Fix $\xx\in V_{\alpha\beta}$. For a tree edge
$e=\{m,n\}\in \edge(\tree_\alpha)$, using its active direction $\ell_{\alpha,e}$, the $(m,n)$-entry gives
$|1-r_mr_n|\, \left| (\tilde{\calA}_1^{(\alpha)\ell_{\alpha,e}})_{mn}(\xx) \right| \le 2\varepsilon$. 
By the active lower bound, $\left| (\tilde{\calA}_1^{(\alpha)\ell_{\alpha,e}})_{mn}(\xx) \right| \ge \kappa_0$. 
If $r_mr_n=-1$, then $|1-r_mr_n|=2$, and the preceding inequality implies $2\kappa_0\le 2\varepsilon$, contradicting $\varepsilon<\kappa_0$. Hence $r_mr_n=1$
for every edge $e=\{m,n\}\in\edge(\tree_\alpha)$. Since $\tree_\alpha$ is connected, all $r_m$ are equal. Therefore $R=\pm I$. 
Equivalently, $S_\alpha\widehat S_\beta^{(\alpha)}=\pm I$, and hence $\widehat S_\beta^{(\alpha)}=\pm S_\alpha$. 
That is,
\begin{equation*}
    T_2^{(\alpha\beta)}S_\beta T_1^{(\alpha\beta)} = \pm S_\alpha. 
\end{equation*}

Finally, on $U_\alpha\cap U_\beta$,
\begin{equation*}
    \begin{aligned}
    Q^{(\beta)}
    &= \eframe_1^{(\beta)} S_\beta \eframe_2^{(\beta)\trans}  
    = \eframe_1^{(\alpha)} T_1^{(\alpha\beta)} S_\beta T_2^{(\alpha\beta)} \eframe_2^{(\alpha)\trans}.
\end{aligned}
\end{equation*}
Since all sign matrices commute,
\begin{equation*}
    T_1^{(\alpha\beta)}S_\beta T_2^{(\alpha\beta)} = T_2^{(\alpha\beta)}S_\beta T_1^{(\alpha\beta)} = \pm S_\alpha. 
\end{equation*}
Therefore
\begin{equation*}
    Q^{(\beta)} = \pm \eframe_1^{(\alpha)} S_\alpha \eframe_2^{(\alpha)\trans} = \pm Q^{(\alpha)}.
    \qedhere
\end{equation*}
\end{proof}

Overlap compatibility can be propagated along paths in $K$: the signs of the local representatives can be chosen successively so that they glue together continuously. \Cref{lem:pathwise_projective_constancy} shows that this pathwise lift, combined with the local derivative bound, yields an almost-constancy estimate.

\begin{lemma}[Pathwise projective constancy]
\label{lem:pathwise_projective_constancy}
Let $K\subset\R^{\dof}$ be compact. Assume that there exist $\xx_\ast\in K$ and $L_K<\infty$ such that, for every $\xx\in K$, there is a rectifiable path $\gamma:[0,1]\to K$ with $\gamma(0)=\xx_\ast$ and $\gamma(1)=\xx$, of length at most $L_K$. Let $\{U_\alpha\}_{\alpha=1}^{N_{\rm cov}}$ be a finite open cover of $K$, and let $Q^{(\alpha)}:U_\alpha\to \ON$ be $C^1$ maps. Assume that the following two conditions hold.

\begin{enumerate}[label=(\roman*), leftmargin=*, widest=iii]
\item \emph{Overlap compatibility up to sign:} whenever
$U_\alpha\cap U_\beta\cap K\neq\emptyset$, there exists $\eta_{\alpha\beta}\in\{\pm1\}$ such that
$Q^{(\beta)} = \eta_{\alpha\beta}Q^{(\alpha)}$ on $U_\alpha\cap U_\beta\cap K$. 

\item \emph{Derivative bound:} for some $\Gamma_Q>0$, 
    $\max_{\alpha} \max_{1\le\mu\le\dof} \|\partial_{x_\mu}Q^{(\alpha)}\|_{C^0(U_\alpha\cap K)} \le \Gamma_Q.$
\end{enumerate}

Fix an index $\alpha_\ast$ such that $\xx_\ast\in U_{\alpha_\ast}$, and set $\QO:=Q^{(\alpha_\ast)}(\xx_\ast)$. Then, for every $\xx\in U_\alpha\cap K$,
$d_\pm\bigl(Q^{(\alpha)}(\xx),\QO\bigr) \le \sqrt{\dof}\,\Gamma_Q L_K$, 
where $d_\pm(Q_1,Q_2) := \min_{\eta\in\{\pm1\}}\|Q_1-\eta Q_2\|_F$. 
\end{lemma}

\begin{proof}
Fix $\xx\in U_\alpha\cap K$. Choose a rectifiable path $\gamma:[0,1]\to K$ with $\gamma(0)=\xx_\ast$ and $\gamma(1)=\xx$ of length at most $L_K$, parameterized proportionally to arclength. Then $\gamma$ is Lipschitz, hence differentiable almost everywhere, and $|\gamma'(t)| \leq L_K$.

Since $\gamma([0,1])$ is compact and $\{U_\alpha\}_{\alpha=1}^{N_{\rm cov}}$ is a finite open cover of $K$, the Lebesgue number lemma allows us to choose
a subdivision $0=t_0<t_1<\cdots<t_J=1$ and indices $\alpha_1,\ldots,\alpha_J$ such that
$\gamma([t_{j-1},t_j])\subset U_{\alpha_j}$ for $j=1,\ldots,J$. 
Refining the subdivision if necessary, we may assume that $\alpha_1=\alpha_\ast$ and $\alpha_J=\alpha$. 

We now choose signs along this path. Set
$\eta_1:=1$. Suppose that $\eta_j\in\{\pm1\}$ has already been chosen. Since
$\gamma(t_j)\in U_{\alpha_j}\cap U_{\alpha_{j+1}}\cap K,$
the overlap compatibility assumption~(i) gives
$Q^{(\alpha_{j+1})}(\gamma(t_j)) = \pm Q^{(\alpha_j)}(\gamma(t_j))$. 
Thus we can choose $\eta_{j+1}\in\{\pm1\}$ so that
$\eta_{j+1} Q^{(\alpha_{j+1})}(\gamma(t_j)) = \eta_j Q^{(\alpha_j)}(\gamma(t_j))$. 

Define a curve $\widehat Q:[0,1]\to \ON$ by
$\widehat Q(t) := \eta_j Q^{(\alpha_j)}(\gamma(t))$ for $t\in[t_{j-1},t_j]$.
By the choice of the signs $\eta_j$, the curve $\widehat Q$ is continuous at each switching point $t_j$. Hence $\widehat Q$ is continuous
on $[0,1]$. Moreover,
$\widehat Q(0)=\QO$ and $\widehat Q(1)=\eta_J Q^{(\alpha)}(\xx)$. 

On each subinterval $[t_{j-1},t_j]$, the chain rule gives, for almost every
$t$,
\begin{equation*}
    \frac{\dd}{\dd t}\widehat Q(t) = \eta_j DQ^{(\alpha_j)}(\gamma(t))\,\gamma'(t).
\end{equation*}
Using the coordinate derivative bound, we have
\begin{equation*}
    \left\| \frac{\dd}{\dd t}\widehat Q(t) \right\|_F \le \sqrt{\dof}\,\Gamma_Q\,|\gamma'(t)|
    \leq \sqrt{\dof}\,\Gamma_Q L_K \qquad \text{for a.e. $t$,}
\end{equation*}
and it follows readily that $\|\widehat Q(1)-\widehat Q(0)\|_F \leq \sqrt{\dof}\,\Gamma_Q L_K$. 
Since $\widehat Q(1)=\eta_J Q^{(\alpha)}(\xx)$,
we obtain 
\begin{equation*}
    d_\pm\bigl(Q^{(\alpha)}(\xx),\QO\bigr) \le \sqrt{\dof}\,\Gamma_Q L_K. 
    \qedhere
\end{equation*}
\end{proof}

It remains to convert the $d_\pm$-closeness of the local gauges to a fixed $\QO\in\ON$ into an estimate for $\dist_K(H_1,H_2)$; \Cref{lem:projective_to_global_distance} carries out this comparison.

\begin{lemma}[From local gauges to the global distance]
\label{lem:projective_to_global_distance}
Let $K\subset\R^{\dof}$, and let $H_i:K\to\SymHN$ for $i=1,2$. Assume that $\sup_{\xx\in K}\|H_1(\xx)\|_F\le B_H$. Let $\{U_\alpha\}_{\alpha=1}^{N_{\rm cov}}$ be a finite open cover of $K$. Suppose that on each $U_\alpha$, there are ordered orthonormal $C^1$ eigenframes $\eframe_i^{(\alpha)}$ such that 
\begin{equation*}
    H_i(\xx) = \eframe_i^{(\alpha)}(\xx) \Lambda_i(\xx) \eframe_i^{(\alpha)}(\xx)^\trans, \quad \xx\in U_\alpha\cap K, 
\end{equation*}
where
$\Lambda_i(\xx) := \diag(\lambda_1(\xx;H_i),\ldots,\lambda_\size(\xx;H_i))$. 
Let $S_\alpha\in\constS$, and define
$Q^{(\alpha)}(\xx) := \eframe_1^{(\alpha)}(\xx) S_\alpha \eframe_2^{(\alpha)}(\xx)^\trans$. 
Assume that, for some $\varepsilon\ge0$,
$\max_{\xx\in K}\max_{1\le n\le\size} |\lambda_n(\xx;H_2)-\lambda_n(\xx;H_1)| \le \varepsilon$, and that, for some $\QO\in \ON$ and $\varepsilon_Q\ge0$,
$d_\pm\bigl(Q^{(\alpha)}(\xx),\QO\bigr) \le \varepsilon_Q$ for $\xx\in U_\alpha\cap K$, with $d_\pm$ as in \Cref{lem:pathwise_projective_constancy}.
Then
\begin{equation*}
    \sup_{\xx\in K} \left\| H_2(\xx)-\QO^\trans H_1(\xx)\QO \right\|_F \le \sqrt{\size}\,\varepsilon + 2B_H\varepsilon_Q. 
\end{equation*}
In particular,
$\dist_K(H_1,H_2) \le \sqrt{\size}\,\varepsilon + 2B_H\varepsilon_Q$. 
\end{lemma}

\begin{proof}
Fix $\xx\in K$, and choose $\alpha$ such that $\xx\in U_\alpha$. By the definition of $Q^{(\alpha)}$,
\begin{equation*}
    \begin{aligned}
    (Q^{(\alpha)}(\xx))^\trans H_1(\xx)Q^{(\alpha)}(\xx)
    &= \eframe_2^{(\alpha)}(\xx)S_\alpha\eframe_1^{(\alpha)}(\xx)^\trans \bigl(\eframe_1^{(\alpha)}(\xx)\Lambda_1(\xx)\eframe_1^{(\alpha)}(\xx)^\trans\bigr) \eframe_1^{(\alpha)}(\xx)S_\alpha\eframe_2^{(\alpha)}(\xx)^\trans  \\
    &= \eframe_2^{(\alpha)}(\xx)S_\alpha\Lambda_1(\xx)S_\alpha\eframe_2^{(\alpha)}(\xx)^\trans.
\end{aligned}
\end{equation*}
Since $S_\alpha$ and $\Lambda_1(\xx)$ are diagonal and $S_\alpha^2=I$, we have $S_\alpha\Lambda_1(\xx)S_\alpha=\Lambda_1(\xx)$. 
Hence
\begin{equation*}
    (Q^{(\alpha)}(\xx))^\trans H_1(\xx)Q^{(\alpha)}(\xx) = \eframe_2^{(\alpha)}(\xx)\Lambda_1(\xx)\eframe_2^{(\alpha)}(\xx)^\trans.
\end{equation*}
On the other hand, 
\begin{equation*}
    H_2(\xx)=\eframe_2^{(\alpha)}(\xx)\Lambda_2(\xx)\eframe_2^{(\alpha)}(\xx)^\trans.
\end{equation*}
By orthogonal invariance of the Frobenius norm,
\begin{equation*}
    \begin{aligned}
    \|H_2(\xx)-(Q^{(\alpha)}(\xx))^\trans H_1(\xx)Q^{(\alpha)}(\xx)\|_F
    &= \|\eframe_2^{(\alpha)}(\xx)(\Lambda_2(\xx)-\Lambda_1(\xx))\eframe_2^{(\alpha)}(\xx)^\trans\|_F  \\
    &= \|\Lambda_2(\xx)-\Lambda_1(\xx)\|_F 
    \le \sqrt{\size}\,\varepsilon.
\end{aligned}
\end{equation*}

Next, by the assumed bound on $d_\pm$, there exists
$\eta(\xx)\in\{\pm1\}$ such that
\begin{equation*}
    \|Q^{(\alpha)}(\xx)-\eta(\xx)\QO\|_F\le \varepsilon_Q. 
\end{equation*}
Since a global sign cancels under conjugation,
\begin{equation*}
    (\eta(\xx)\QO)^\trans H_1(\xx)(\eta(\xx)\QO) = \QO^\trans H_1(\xx)\QO.
\end{equation*}
We estimate
\begin{equation*}
    \begin{aligned}
&\left\| (Q^{(\alpha)}(\xx))^\trans H_1(\xx)Q^{(\alpha)}(\xx) - (\eta(\xx)\QO)^\trans H_1(\xx)(\eta(\xx)\QO) \right\|_F  \\
= & \left\| (Q^{(\alpha)}(\xx))^\trans H_1(\xx)(Q^{(\alpha)}(\xx)-\eta(\xx)\QO) + (Q^{(\alpha)}(\xx)-\eta(\xx)\QO)^\trans H_1(\xx)(\eta(\xx)\QO) \right\|_F  \\
\le & \|(Q^{(\alpha)}(\xx))^\trans H_1(\xx)(Q^{(\alpha)}(\xx)-\eta(\xx)\QO)\|_F + \|(Q^{(\alpha)}(\xx)-\eta(\xx)\QO)^\trans H_1(\xx)(\eta(\xx)\QO)\|_F  \\
\le & 2\|H_1(\xx)\|_F\|Q^{(\alpha)}(\xx)-\eta(\xx)\QO\|_F  \\
\le & 2B_H\varepsilon_Q.
\end{aligned}
\end{equation*}
Combining the two bounds gives
\begin{equation*}
    \begin{aligned}
\left\| H_2(\xx)-\QO^\trans H_1(\xx)\QO \right\|_F
&\le \|H_2(\xx)-(Q^{(\alpha)}(\xx))^\trans H_1(\xx)Q^{(\alpha)}(\xx)\|_F  \\
&\quad+ \|(Q^{(\alpha)}(\xx))^\trans H_1(\xx)Q^{(\alpha)}(\xx)-\QO^\trans H_1(\xx)\QO\|_F  \\
&\le \sqrt{\size}\,\varepsilon + 2B_H\varepsilon_Q.
\end{aligned}
\end{equation*}
Taking the supremum over $\xx\in K$ proves the first estimate. Since $\QO\in \ON$
is admissible in the infimum defining $\dist_K$, the second estimate follows.
\end{proof}

We now assemble the preceding lemmas into a proof of \Cref{thm:full_stability}.

\begin{proof}[Proof of \Cref{thm:full_stability}]
Set $\delta:=\Delta_K(H_1,H_2)$. If $\delta=0$, the conclusion follows from \Cref{thm:complete_invariant_new}. 
Hence we assume $\delta>0$. The threshold $\delta_0>0$ of the statement is fixed below through finitely many smallness conditions.

Choose $\delta_0\le g_K/4$. Then \Cref{lem:stab_H2_nondeg} implies that
$H_2$ is non-degenerate on $K$, with spectral gap at least $g_K/2$. Thus,
near every point of $K$, both $H_1$ and $H_2$ admit ordered $C^1$
eigenframes and continuous local derivative couplings.

Fix $\xx_0\in K$, and freeze the pointwise tree and active labels $(\tree(\xx_0),\ell(\xx_0))$. 
At $\xx_0$, the frozen choice agrees with the pointwise choice in
$\Delta_K$, so $\|\Delta\map(\xx_0,\xx_0)\|_\infty\le\delta$. 
Moreover, the first active lower bound required by \Cref{lem:stab_freezing}
holds by the definition of $\kappa_K$ in \eqref{eq:kappa_dir_lower_bound}, and
the second follows from the pure two-loop entries of $\Delta\map(\xx_0,\xx_0)$. 
Choosing $\delta_0\le 7\kappa_K^2/16$, we apply
\Cref{lem:stab_freezing} and obtain, for every $\xx_0\in K$, an open ball
$U_{\xx_0}$ such that
\begin{equation*}
    \|\Delta\map(\cdot,\xx_0)\|_{C^0(U_{\xx_0}\cap K)} \le 2\delta,
\end{equation*}
and
\begin{equation*}
    |\tilde{\calA}_e^{\ell_e(\xx_0)}(\xx,\xx_0;H_1)| \ge \frac{\kappa_K}{2}, \qquad |\tilde{\calA}_e^{\ell_e(\xx_0)}(\xx,\xx_0;H_2)| \ge \frac{\kappa_K}{4}
\end{equation*}
for all $\xx\in U_{\xx_0}$ and all $e\in\edge(\tree(\xx_0))$.

By compactness, choose a finite subcover $K\subset\bigcup_{\alpha=1}^{N_{\rm cov}}U_\alpha$. 
For the centre $\xx_\alpha$ of $U_\alpha$, write $\tree_\alpha:=\tree(\xx_\alpha)$, $\ell_{\alpha,e}:=\ell_e(\xx_\alpha)$, $\eframe_i^{(\alpha)}(\xx):=\eframe_i(\xx,\xx_\alpha)$, and $\tilde{\calA}_i^{(\alpha)\mu}(\xx):=\tilde{\calA}^{\mu}(\xx,\xx_\alpha;H_i)$.
On $U_\alpha$, set $A_i^\mu(\xx) := \tilde{\calA}_i^{(\alpha)\mu}(\xx)$.
The hypotheses of \Cref{lem:stab_local_reconstruction} are satisfied with $U=U_\alpha$, $V=U_\alpha\cap K$, $\varepsilon=2\delta$, $\kappa=\frac{\kappa_K}{4}$, 
provided $\delta_0\le 1/2$. Hence, for each $\alpha$, there exists
$S_\alpha\in\constS$ such that
\begin{equation*}
    \max_{1\le\mu\le\dof} \left\| \tilde{\calA}_2^{(\alpha)\mu} - S_\alpha \tilde{\calA}_1^{(\alpha)\mu} S_\alpha \right\|_{C^0(U_\alpha\cap K)} \le C\delta.
\end{equation*}

Define the local orthogonal gauges $Q^{(\alpha)}(\xx) := \eframe_1^{(\alpha)}(\xx) S_\alpha \eframe_2^{(\alpha)}(\xx)^\trans$. 
On any nonempty overlap
$V_{\alpha\beta}:=U_\alpha\cap U_\beta\cap K$, 
the preceding estimate holds for both $\alpha$ and $\beta$. Moreover, the
active lower bound on $U_\alpha$ gives
$|(\tilde{\calA}_1^{(\alpha)\ell_{\alpha,e}})_e(\xx)| \ge \frac{\kappa_K}{2}$.
Choosing $\delta_0$ smaller if necessary so that $C\delta_0<\kappa_K/2$,
we may apply \Cref{lem:stab_overlap_compatibility}. Therefore
\begin{equation*}
    Q^{(\beta)}=\pm Q^{(\alpha)} \qquad \text{on }U_\alpha\cap U_\beta.
\end{equation*}

Next, differentiating
$Q^{(\alpha)} = \eframe_1^{(\alpha)}S_\alpha\eframe_2^{(\alpha)\trans}$
and using the estimate above gives
\begin{equation*}
    \max_\alpha\max_{1\le\mu\le\dof} \|\partial_{x_\mu}Q^{(\alpha)}\|_{C^0(U_\alpha\cap K)} \le C\delta.
\end{equation*}
Together with the overlap compatibility, \Cref{lem:pathwise_projective_constancy}
applies with
$\Gamma_Q=C\delta$. 
Thus there exists $\QO\in\ON$ such that, for every
$\xx\in U_\alpha\cap K$,
\begin{equation*}
    d_\pm\bigl(Q^{(\alpha)}(\xx),\QO\bigr) \le C L_K\delta.
\end{equation*}

Finally, the spectral component of $\map_{\tree,\ell}$ gives
\begin{equation*}
    \max_{\xx\in K}\max_{1\le n\le\size} |\lambda_n(\xx;H_2)-\lambda_n(\xx;H_1)| \le \delta.
\end{equation*}
Applying \Cref{lem:projective_to_global_distance} with
$\varepsilon=\delta$, $\varepsilon_Q=C L_K\delta$, $B_H=\sup_{\xx\in K}\|H_1(\xx)\|_F\le B_K$, we obtain
\begin{equation*}
    \sup_{\xx\in K} \|H_2(\xx)-\QO^\trans H_1(\xx)\QO\|_F \le C\delta.
\end{equation*}
Since $\QO\in\ON$, this gives
\begin{equation*}
\dist_K(H_1,H_2) \le C\delta = C\Delta_K(H_1,H_2).
\qedhere
\end{equation*}
\end{proof}

\bibliography{refs} 

\end{document}